%% file: AH-Lp-rigidity.tex
\documentclass[a4paper,reqno]{amsart}
\usepackage{enumerate,amsmath,amssymb,stmaryrd} 
\usepackage{pstricks,pst-node,pst-coil,pst-plot} 
\usepackage{hyperref}
\usepackage{euscript}
\usepackage{verbatim}
\usepackage{graphicx}

\numberwithin{equation}{section}

\input{def}

\def\begeq{\begin{equation}}
\def\endeq{\end{equation}}

\def\tr{{\rm tr}}

\begin{document}

\title[Einstein manifolds with Weyl curvature in $L^p$]
 {On the asymptotic behavior of Einstein manifolds with an integral bound on the Weyl curvature}

\begin{abstract}
In this paper we consider the geometric behavior near infinity of some Einstein manifolds $(X^n, g)$
with Weyl curvature belonging to a certain $L^p$ space. Namely, we show that if $(X^n, g)$, $n \geq 7$,
admits an essential set and has its Weyl curvature in $L^p$ for some $1<p<\frac{n-1}{2}$, then $(X^n, g)$
must be asymptotically locally hyperbolic. One interesting application of this theorem is to show a
rigidity result for the hyperbolic space under an integral condition for the curvature.
\end{abstract}


\author[R. Gicquaud, D. Ji and Y. Shi]{Romain Gicquaud, Dandan Ji$^\dag$ and Yuguang Shi$^\dag$}

\address{Romain Gicquaud, Laboratoire de Math\'ematiques et de Physique Th\'eorique,
UFR Sciences et Technologie, Universit\'e Fran\c cois Rabelais, Parc de Grandmont,
37300 Tours, France} \email{romain.gicquaud@lmpt.univ-tours.fr}

\address{Dandan Ji, Key Laboratory of Pure and Applied mathematics, School of Mathematics Science, Peking University,
Beijing, 100871, P.R. China.} \email{jidandan@pku.edu.cn}

\address{Yuguang Shi, Key Laboratory of Pure and Applied mathematics, School of Mathematics Science, Peking University,
Beijing, 100871, P.R. China.} \email{ygshi@math.pku.edu.cn}

\thanks{$^\dag$ Research partially supported by   NSF grant of China 10990013.}

\date{\today}
\maketitle
\tableofcontents

\section{Introduction}
\label{secIntro}
\input{intro}

\section{Basic Estimates}
\label{secEstimates}
\input{basic}

\section{Pointwise estimate for the Weyl tensor}
\label{secEstimates2}
\input{pointwise}

\section{Applications}
\label{secApplications}
\input{applications}

\bibliographystyle{amsplain}
\bibliography{biblio-AH-Lp-rigidity}

\end{document}

%% file: def.tex
\theoremstyle{plain}

\newtheorem{thm}{Theorem}[section]

\newtheorem{lem}[thm]{Lemma}
\newtheorem{lemma}[thm]{Lemma}
\newtheorem{prop}[thm]{Proposition}
\newtheorem{proposition}[thm]{Proposition}

\theoremstyle{remark}

\newtheorem{claim}[thm]{Claim}
\newtheorem{rem}[thm]{Remark}     

\theoremstyle{definition}
\newtheorem{defn}[thm]{Definition}

\newcounter{mnotecount}[section]

\newcommand{\definedas}{\mathrel{\raise.095ex\hbox{\rm :}\mkern-5.2mu=}}

\def\epsilon{{\varepsilon}}
\def\phi{{\varphi}}

\let\<\langle 
\let\>\rangle

\newcommand{\Sigmatil}{\widetilde{\Sigma}}

\newcommand{\bR}{\mathbb{R}}

\newcommand{\bS}{\mathbb{S}}

\newcommand{\bD}{\mathbb{D}}

\newcommand{\cD}{\mathcal{D}}

\newcommand{\cB}{\mathcal{B}}
\newcommand{\cQ}{\mathcal{Q}}

\newcommand{\bfB}{\mathbf{B}}
\newcommand{\bfK}{\mathbf{K}}

\newcommand{\gbar}{\overline{g}}
\newcommand{\ghat}{\widehat{g}}

\renewcommand{\hbar}{\overline{h}}

\newcommand{\Wtil}{\widetilde{W}}

\newcommand{\kulk}{\owedge} 

\DeclareMathOperator{\tr}{tr}
\DeclareMathOperator{\divg}{div}

\DeclareMathOperator{\supp}{supp}
\DeclareMathOperator{\vol}{Vol}

\newcommand{\diff}  [2]{\frac{d #1}{d #2}}

\newcommand{\riem}{\mathrm{Rm}}
\newcommand{\hriem}{\widehat{\riem}}
\newcommand{\riemdddd}[4]{\riem_{#1 #2 #3 #4}}
\newcommand{\riemuddd}[4]{\riem^{#1}_{\phantom{#1} #2 #3 #4}}

\newcommand{\riemudud}[4]{\riem^{#1 \phantom{#2} #3}_{\phantom{#1} #2 \phantom{#3} #4}}
\newcommand{\riemdudu}[4]{\riem^{\phantom{#1} #2 \phantom{#3} #4}_{#1 \phantom{#2} #3}}
\newcommand{\riemuddu}[4]{\riem^{#1 \phantom{#2 #3} #4}_{\phantom{#1} #2 #3}}

\newcommand{\riemduud}[4]{\riem^{\phantom{#1} #2 #3}_{#1 \phantom{#2 #3} #4}}

\newcommand{\riembar}{\overline{\riem}}

\newcommand{\weyl}{\mathrm{W}}
\newcommand{\weylbar}{\overline{\weyl}}
\newcommand{\hweyl}{\widehat{\weyl}}
\newcommand{\weyldddd}[4]{\weyl_{#1 #2 #3 #4}}
\newcommand{\weyluddd}[4]{\weyl^{#1}_{\phantom{#1} #2 #3 #4}}

\newcommand{\weyludud}[4]{\weyl^{#1 \phantom{#2} #3}_{\phantom{#1} #2 \phantom{#3} #4}}

\newcommand{\ric}{\mathrm{Ric}}

\newcommand{\ricud}[2]{\ric^{#1}_{\phantom{#1} #2}}

\newcommand{\scal}{\mathrm{Scal}}

\newcommand{\tudud}[5]{\mathcal{#1}^{#2 \phantom{#3} #4}_{\phantom{#2} #3 \phantom{#4} #5}}

\newcommand{\grad}[1]{\nabla_{#1}}

\DeclareMathOperator{\hess}{Hess}
\newcommand{\hessdd}[2]{\nabla^2_{#1, #2}}

\newcommand{\sff}{S}

\newcommand{\sffud}[2]{\sff^{#1}_{\phantom{#1} #2}}

\newcommand{\kronecker}[2]{\delta^{#1}_{\phantom{#1}#2}}


%% file: intro.tex
During the last three decades there were lots of interesting works on the asymptotic behavior of
Ricci flat metrics with integral bounds on curvature. See e.g. \cite{BKN} and \cite{CT}. These works
gave an nice intrinsic characterization of asymptotic locally Euclidean (ALE) manifolds. Inspired
by these works, we want to study a similar problem in the context of asymptotic locally hyperbolic
(ALH) manifolds. The ALH case appears much more complicated than the ALE case in both geometric and
analytic parts. As an example, the rescaling argument which is very efficient in the analysis of
the asymptotic behavior of ALE metrics does not work in the ALH case because the model is the
hyperbolic metric which is not scale invariant. Another complication arises from the M\"obius group
of $\bS^n$ which allows cuspidal ends. To rule out such ends, we need to assume that the manifold
$(X^{n}, g)$ admits an essential set (see \cite{BM} for more details):

\begin{defn}\label{essentialset}
A non empty compact subset $\bD$ of a complete noncompact Riemannian manifold $(X^{n}, g)$ is called
an \emph{essential set} if

\begin{enumerate}
 \item $\bD$ is a compact domain of $X^n$ with smooth and strictly convex boundary
 $\bfB \definedas \partial \bD$, i.e. its second fundamental forms with respect to the outward unit
 normal vector field is positive definite,
 \item $\bD$ is totally convex: there is no geodesic $\gamma: [a; b] \to X$ such that
 $\gamma(a), \gamma(b) \in \bD$ and $\gamma(c) \not\in \bD$ for some $c \in [a; b]$,
 \item the sectional curvature of $(X^n, g)$ is negative outside $\bD$.
\end{enumerate}
\end{defn}

Assuming that $(X^n, g)$ is hyperbolic the existence of an essential set is equivalent to the
requirement that $(X^n, g)$ is convex and co-compact. More generally, it can be shown that
conformally compact and Cartan-Hadamard manifolds admit essential sets, see \cite{GicquaudThesis}.
Together with assumptions on the rate of convergence of the sectional curvature to $-1$ at infinity,
the existence of an essential set has been used to prove the existence and the regularity of a
conformal compactification of the manifold $(X^n, g)$ in \cite{Ba, BG, GicquaudCompactification, HQS}.

In what follows we define $\rho : X \to \bR$ as the distance function from $\bD$:
$$\rho \definedas d_g(\bD, \cdot).$$

In \cite{BM}, it has been proven that, if $\bD \subset X$ is an essential set, $\rho$ is smooth
function and has no critical point. This implies that the region $X^n$ which is outside the
essential set $\bD$ is diffeomorphic to $[0,\infty)\times \bfB$.

In this article we want to investigate the behavior at infinity of some Einstein manifolds with
Weyl curvature belonging to a certain $L^p$ space. In particular, we show that they are
asymptotically locally hyperbolic Einstein (ALHE) metric outside the essential set $\bD$, meaning
that $\sec_g + 1 = O(e^{-a\rho})$ for some $a > 0$. In contrast with the ALE case, the major
difficulty in the ALH setting is the lack of sharp global Sobolev inequalities which are crucial
in applying Moser iterations in the ALE case (see e.g. \cite{BKN}, \cite{CT}).

However, we observed a nice $L^2$-estimate for the Laplace operator acting on 4-tensors satisfying
properties analogous to those of the Weyl tensor for manifolds of dimension greater than $5$, see
Lemma \ref{q} below. Thanks to this lemma and combining other techniques, we were able to obtain
the following result (See also Theorems \ref{mainthm}):

\begin{thm}\label{mainresult1}
Let $(X^n,g)$, $n \geq 7$, be a complete noncompact Einstein manifold with
$$\ric = -(n-1) g.$$
Assume that $X^n$ contains an essential set $\bD$. We denote $W$ the Weyl tensor of the metric
$g$. If $\| W\|_{L^p(X^n,g)} <\infty$ for some $p \in \left(1; \frac{n-1}{2}\right)$, then
there exists a constant $C$ such that

\begin{equation}\label{curvaturebehavioratinfinity}
 \left| \riem - \bfK \right| \leq C e^{-(n+1)\rho}.
\end{equation}

Here $\riem$ is the curvature tensor of the metric $g$ and $\bfK$ the constant curvature tensor
with sectional curvature $-1$ with respect to metric $g$, i.e. 
$$\bfK_{ijkl}=-\left(g_{ik}g_{jl}-g_{ij}g_{kl}\right).$$
\end{thm}

Since $(X^n, g)$ is Einstein and has a lower bound on its injectivity radius, it will become apparent
that $\weyl \in L^\infty$. As a consequence if $\weyl \in L^p$ for some $p \in (1, \infty)$,
$\weyl \in L^q$ for all $q \geq p$: the smaller $p$ is, the more stringent the assumption.

This result turns out to be very useful to prove rigidity theorems. In particular, assuming further
that the manifold $X$ is simply connected at infinity forces $(X, g)$ to be isometric to the
hyperbolic space (see Theorem \ref{rigidity}). We also give a variant of this theorem for static
spacetimes together with a rigidity result in Section \ref{secApplications}.

We are interested in this article only in complete noncompact manifolds whose curvature will be
shown to tend to $-1$ at infinity. Hence, we will always use the shorthand ``\emph{Einstein manifold}''
to denote manifolds $(X^n, g)$ satisfying \[\ric_g = -(n-1)g.\]

Einstein metrics constructed in \cite{GL}, \cite{Le} and \cite{Biquard} satisfy $|\weyl| \leq C e^{-2\rho}$.
In particular $\weyl \in L^p$ for any $p > \frac{n-1}{2}$. The case $p = \frac{n-1}{2}$ is more
delicate and we plan to address it in a future work. Nevertheless Theorem \ref{curvaturebehavior}
shows that Theorem \ref{mainresult1} remains true for $n \geq 5$ and $p \leq \frac{n-1}{2}$ in the
important case of conformally compact metrics.

In the ALE case, the curvature behavior at infinity which is the analog of
\eqref{curvaturebehavioratinfinity} is obtained by a Moser iteration where a global Sobolev inequality
is involved. However, as mentioned above, in the ALH case such a kind of global Sobolev inequality
is not true (see \cite{GicquaudStability} for an illustration of this fact). Hence, we use a variant
of the maximum principle to get Estimate \eqref{curvaturebehavioratinfinity}. This is where the
assumption $n \geq 7$ appears. It is also interesting to compare Theorem \ref{mainresult1} with
\cite[Theorem 0.13]{CT} and \cite[Theorem 1.5]{BKN}.

Let us describe the main arguments that lead to Theorem \ref{mainresult1}. First, we note that if
$(X^n,g)$ is Einstein, its Weyl tensor satisfies the following well known equation:

\begin{equation}\label{equationofweyltensor}
\triangle W + 2(n-1) W + 2\cQ(W) = 0,
\end{equation}
where $\triangle$ is the Laplace operator acting on tensors and $\cQ$ is a quadratic expression
in the Weyl curvature tensor. See e.g. \cite{BG} for a derivation of this formula. Setting
\begin{equation}\label{eqDefB}
\cB_{\alpha\beta\gamma\delta} \definedas \weyludud{\mu}{\alpha}{\nu}{\beta} \weyldddd{\mu}{\gamma}{\nu}{\delta},
\end{equation}
$\cQ$ can be written as follows:

\begin{equation}\label{eqDefQ}
\cQ_{\alpha\beta\gamma\delta} \definedas \cB_{\alpha\beta\gamma\delta} + \cB_{\alpha\gamma\beta\delta}-
\cB_{\beta\alpha\gamma\delta} - \cB_{\beta\delta\alpha\gamma}.
\end{equation}

Note that we are using the Einstein summation convention. Due to the $L^p$-bound of $W$, we see that
$W$ is small near infinity. Hence, intuitively Equation \eqref{equationofweyltensor} is almost
equivalent to the following linear equation:

\begin{equation}\label{equationofweyltensor2}
\triangle W + 2(n-1) W  = 0.
\end{equation}

By some careful analysis, we are able to show an $L^2$ spectral estimate of the Laplace operator acting
on Weyl-type tensors (see Lemma \ref{q}). Together with a refined Kato inequality and other some other
techniques we achieve the proof of Estimate \eqref{curvaturebehavioratinfinity}.

Some applications of Theorem \ref{mainresult1} are considered in this paper. Namely, by Theorem
\ref{mainresult1} we are able to show a rigidity theorem for ALHE manifolds with Weyl tensor belonging
to $L^p$. We also get the curvature behavior of vacuum static spacetimes with a negative cosmological
constant. See Theorems \ref{rigidity} and \ref{static} for more details.

The rest of the paper goes as follows. In \S\ref{secEstimates} we get some basic $L^2$-estimates
for the Weyl tensor. Then we show how these estimates can be converted to pointwise estimates in
\S\ref{secEstimates2}. Finally in \S\ref{secApplications}, we discuss some applications of Theorem
\ref{mainresult1}.\\

{\bf Acknowledgements} The authors are grateful to Professor Jie Qing, Dr. Jie Wu and Dr. Xue Hu for
their interest in this work and for many enlightening discussions.

%% file: basic.tex
The main purpose of this section is to prove Lemma \ref{spetrum} which gives a spectral estimate
of some $(0,4)$-tensors on asymptotically hyperbolic Einstein (AHE) manifolds with an essential
set. It will play an essential role in the proof of the main theorem.

In particular, Lemma \ref{spetrum} will be used to show that if $\|\weyl\|_{L^p(X^n,g)} < \infty$ for
some $p \in \left(1, \frac{n-1}{2}\right)$, then $\|\weyl\|_{L^2(X^n,g)} <\infty$. Moreover if
$n \geq 6$, we have $\|e^{\frac{a}{2}\rho}\weyl\|_{L^2(X^n,g)} <\infty$ for some positive $a$.
See Proposition \ref{lp} for more details.

We choose once and for all a complete noncompact Einstein manifold $(X^n, g)$ containing an
essential subset $\bD$. We first introduce some estimates for a Riccati equation that will be
useful for the analysis of the normal curvature equation. Similar results have been obtained in
\cite[Lemma 2.3]{ST}. See also \cite{Ba,BG,GicquaudCompactification,HQS}.

In all this section, we use Greek letters to denote indices going from $0$ to $n-1$ and Latin
letters for indices from $1$ to $n-1$. Unless otherwise stated, we use the Einstein summation
convention. For any $R \geq 0$, we denote \[\bD_R \definedas \{x \in X, d_g(x, \cD) \leq R\},\]
and \[\Sigma_R = \rho^{-1}(R)\] a slice of constant $\rho$.

\begin{lem}\label{estfrriccatiequa} Let $\epsilon$ be a positive constant.
Assume that $f(\rho)$ is a smooth positive function of $\rho > 0$ such that $|f(\rho)-1|\leq \epsilon$.
Assume further that $y$ is a solution of
\[y' + y^2 = f\]
satisfying
\[y(0) > 0.\]
Then $y$ satisfies
$$|y-1|\leq e^{-\rho/2}\left|y(0)-1\right| + \epsilon.$$
\end{lem}

\begin{proof}
We claim that $y > 0$. Indeed, if there exists $\rho$ such that $y(\rho)\leq 0$, we can find some
$\rho_0$ satisfying $y(\rho_0)=0$ and $y(\tau)>0$ for any $\tau \in (0, \rho_0)$. In particular
this implies that $y'(\rho_0)\leq0$. But $$0\geq y'(\rho_0)+y^2(\rho_0)=f(\rho_0)>0,$$
which is a contradiction. Next, we set $z = y-1$. The equation satisfied by $y$ implies the following
one for $z$:
\[(z^2)' + 2(y+1) z^2 = 2z(f-1).\]
In particular, since we noticed that $y > 0$, we have
\[(z^2)' + 2 z^2 < 2z(f-1) \leq z^2 + (f-1)^2.\]
This inequality can be integrated to yield
\[|z| \leq \sqrt{z^2(0) e^{-\rho} + \epsilon^2} \leq |z|(0) e^{-\rho/2} + \epsilon.\]
\end{proof}

We want to find nice coordinate charts to apply Schauder estimates. We choose to use harmonic
coordinates. We refer the reader to \cite{HH} and references therein for more informations. Let
$Q > 1$ and $\alpha \in (0; 1)$ be arbitrary. Since the injectivity radius $r_I$ of $(X^n, g)$
is strictly positive (see Lemma \ref{lmInjRadius} below), there exists a constant $r_H > 0$ such
that, given any point $x_0\in X^n$, there exist harmonic coordinates $y^1,...,y^n$ on the ball
$B_{r_H}(x_0)$ in which the metric $g$ satisfies

\begin{equation*}
\left\{
\begin{aligned}
 Q^{-1}\delta & \leq g\leq Q\delta,\\
 \left\|g-\delta\right\|_{C^{1,\alpha}} & \leq Q-1,
\end{aligned} \right.
\end{equation*}
where $\delta = dy^1\otimes dy^1+...+dy^n\otimes dy^n$ is the flat metric.

\begin{lem}\label{lmInjRadius}
 Assume that $(X^n, g)$ satisfies $\|\weyl\|_{L^p(X^n,g)} <\infty$ for some $p \in (1; \infty)$.
 Then the injectivity radius $r_I(x)$ is bounded from below by some positive constant on $(X^n, g)$.
\end{lem}

\begin{proof}
The injectivity radius of $r_I(x)$ a point $x \in X$ is a positive continuous map from $X$ to
$\bR_+^* \cup \{\infty\}$. Hence it is bounded from below on $\bD$ by some $r_0$. Assume that there
exists a point $x \in X\setminus \bD$ whose injectivity radius is less than $\frac{r_0}{2}$. We can
assume that any point $y$ with $\rho(y) < \rho(x)$ has injectivity radius strictly greater than
$\frac{r_0}{2}$. Then there exists a geodesic $\gamma: [0;1] \to X$ of length $r_0$ such that
$\gamma(0) = \gamma(1) = x$.

The function $\rho \circ \gamma$ is convex and cannot be constant. Indeed, if $\rho(\gamma(t)) > 0$,
$(\rho\circ\gamma)''(t) = \sff(\dot{\gamma}, \dot{\gamma}) \geq 0$, with equality iff $\dot{\gamma}$
is colinear to $\nabla \rho$. So $\rho(\gamma(1/2)) < \rho(x)$. Consider now the geodesics
$\gamma_1$ and $\gamma_2$ defined on the interval $[0; 1]$ by

$$
\begin{aligned}
 \gamma_1(t) & = \gamma\left(\frac{1+t}{2}\right),\\
 \gamma_2(t) & = \gamma\left(\frac{1-t}{2}\right).
\end{aligned}
$$

They are two geodesics starting at $\gamma(1/2)$ and ending at $x$, both of length $\frac{r_0}{2}$.
This means that $\gamma(1/2)$ has injectivity radius less than $\frac{r_0}{2}$ and contradicts the
definition of $x$.
\end{proof}

\begin{lem}\label{lmDecayWeyl}
If we further assume that $\|\weyl\|_{L^p(X^n,g)} <\infty$ for some $p \in (1; \infty)$, then the
Weyl tensor $\weyl$ of $(X^n, g)$ tends uniformly to zero at infinity.
\end{lem}

\begin{proof}
In harmonic coordinates, the metric $g$ satisfies an equation of the form
$$\ric_{ij}=-\frac{1}{2}g^{kl}\partial_k\partial_l g_{ij}+Q(g,\partial g)$$
where $Q$ is an expression which is quadratic in $\partial g$, see e.g. \cite{Petersen}. Since $g$
is Einstein, $\ric_{ij} = -(n-1)g_{ij}$, we get by standard elliptic regularity that there exists
a constant $C_1$ such that

$$\left\|g\right\|_{C^{2,\alpha}\left(B_{\frac{2}{3}r_H}(x_0)\right)}\leq C_1.$$

Thus we get a bound

$$\left\|\weyl\right\|_{L^\infty\left(B_{\frac{2}{3}r_H}(x_0)\right)}\leq C_2.$$

From $\weyl\in L^p(X^n,g)$, we get that for any small $\mu > 0$ there exists $R > 0$ such that

$$\left(\int_{X^n\setminus \bD_{R-r_H}}|\weyl|^p dV_g\right)\leq\mu.$$

As a consequence, for any $x_0\in M\setminus \bD_R$, we have that

$$\left\|\weyl\right\|_{L^p\left(B_{\frac{2}{3}r_H}(x_0)\right)} \leq \left\| \weyl\right\|_{L^p\left(X^n\setminus \bD_{R-r_H}\right)}\leq\mu.$$

Select $q\in(\frac{n}{2},\infty)$, $q > p$ arbitrarly. From Young's inequality, there exists
$\beta \in (0, 1)$ such that

$$
\left\|\weyl\right\|_{L^q\left(B_{\frac{2}{3}r_H}(x_0)\right)}
  \leq \left\|\weyl\right\|^\beta_{L^p\left(B_{\frac{2}{3}r_H}(x_0)\right)} \left\|\weyl\right\|^{1-\beta}_{L^\infty\left(B_{\frac{2}{3}r_H}(x_0)\right)}
  \leq C^{1-\beta}_{2}\mu^\beta.
$$

From \cite{BG}, the Weyl tensor satisfies an equation of the form

$$\triangle \weyl+ 2(n-1)\weyl + 2\cQ(\weyl) = 0,$$

\noindent where $\cQ$ was defined in Equation \eqref{eqDefQ}. Therefore from the interior Schauder
estimates, we get

\begin{eqnarray*}
\left\|\weyl\right\|_{W^{2,q}\left(B_{r_{\frac{H}{2}}}(x_0)\right)}
  & \leq & C \left[\left\|\weyl\right\|_{L^q\left(B_{\frac{2}{3}r_H}(x_0)\right)} + \left\|\cQ(\weyl)\right\|_{L^q\left(B_{\frac{2}{3}r_H}(x_0)\right)}\right]\\
  & \leq &C_{3}\mu^\beta,
\end{eqnarray*}
where we used the fact that $\left\|\weyl\right\|_{L^\infty\left(B_{\frac{2}{3}r_H}(x_0)\right)}\leq C_2$ to
estimate the quadratic term.
\end{proof}

For simplicity we may use Fermi coordinates $(x^1, \ldots, x^{n-1})$ on the slices $\Sigma_\rho$.
We denote $\sff_{ij}$ the components of the second fundamental form of $\Sigma_\rho$ in this
coordinate system. It is well known that the following equation holds:

\begin{equation}\label{eqRiccati}
\frac{\partial}{\partial \rho} \sffud{j}{i} + \sffud{j}{k} \sffud{k}{i} = -\riemuddd{j}{0}{i}{0},
\end{equation}

where the index $0$ refers to the unit normal direction of $\Sigma_\rho$, that is to say $\nabla \rho$.
We define the mean curvature of $\Sigma_\rho$ by $H=g^{ij}S_{ij} = \sffud{i}{i}$. Since $g$ is Einstein
with scalar curvature $-n(n-1)$, the Riemann tensor can be written as follows:

\begin{equation}\label{eqRiemEinstein}
\riemdddd{\alpha}{\beta}{\gamma}{\delta}
  = - \left(g_{\alpha\gamma} g_{\beta\delta} - g_{\alpha\delta} g_{\beta\gamma}\right)
  + \weyldddd{\alpha}{\beta}{\gamma}{\delta}.
\end{equation}

Combining Equations \eqref{eqRiccati} and \eqref{eqRiemEinstein} with Lemma \ref{lmDecayWeyl}, we get
the following lemma:

\begin{lemma}\label{lmMeanCurvEstimate}
 Let $H$ be the mean curvature of the hypersurfaces of constant $\rho$.
 If $\|\weyl\|_{L^p(X^n,g)} <\infty$ for some $p \in (1; \infty)$, then $H = (n-1) + o(1)$.
\end{lemma}

\begin{proof}
We fix an arbitrary $\epsilon > 0$. From Equation \eqref{eqRiemEinstein}, the Riccati equation for
the Weingarten operator \eqref{eqRiccati} can be rewritten as follows:
\[
\frac{\partial}{\partial \rho} \sffud{j}{i} + \sffud{j}{k} \sffud{k}{i} = \kronecker{j}{i} - \weyluddd{j}{0}{i}{0}.
\]

From Lemma \ref{lmDecayWeyl}, there exists $\rho_0 > 0$ such that $\left|\weyl\right| < \epsilon$
on $X \setminus \bD_{\rho_0}$. It follows from standard methods (see e.g. \cite[Chapter 6]{Petersen})
together with Lemma \ref{estfrriccatiequa} that $\sff$ satisfies
\[
\left|\sff - \delta\right| \leq \left(\sup_{\Sigma_{\rho_0}} \left|\sff-\delta\right|\right) e^{-(\rho-\rho_0)/2} + \epsilon
\]
on $X^n \setminus \bD{\rho_0}$. In particular, $H = \tr(S)$ is controlled at infinity:
\[
\left|H - (n-1)\right| \leq (n-1) \left(\sup_{\Sigma_{\rho_0}} \left|\sff-\delta\right|\right) e^{-(\rho-\rho_0)/2} + (n-1)\epsilon.
\]
Since $\epsilon$ was arbitrary, this proves the lemma.
\end{proof}

As a consequence of this lemma, we get the following $L^2$-estimate:

\begin{lemma}[Cheng-Yau estimate]\label{lmChengYau}
Assume that $\|\weyl\|_{L^p(X^n,g)} <\infty$ for some $p \in (1; \infty)$. For every $\epsilon > 0$, there
exists a compact subset $K_\epsilon \supset \bD$ such that for any $u \in C^\infty_c(X\setminus K_\epsilon)$,
\[-\int_X u \triangle u \,dV_g \geq \left[\frac{(n-1)^2}{4}-\epsilon\right] \int_X u^2 dV_g.\]
\end{lemma}

\begin{proof}
We set $\phi = e^{-\frac{n-1}{2} \rho}$. Remark that if $\rho_0$ is large enough,
$H \geq (n-1) - \frac{2\epsilon}{n-1}$ on $X \setminus \bD_{\rho_0}$:
\[
-\triangle \phi = -\frac{(n-1)^2}{4} \phi + \frac{n-1}{2} H \phi \geq \left(\frac{(n-1)^2}{4} - \epsilon\right) \phi.
\]

We rewrite

\[
\begin{aligned}
\triangle u & = \triangle \left(\phi \frac{u}{\phi}\right)\\
            & = \frac{\triangle \phi}{\phi} u + 2 \left\<d\phi, d\frac{u}{\phi}\right\> + \phi \triangle \frac{u}{\phi}.
\end{aligned}
\]

So,

\[
\begin{aligned}
- \int_{X \setminus \bD_{\rho_0}} u \triangle u\, dV_g
  & = - \int_{X \setminus \bD_{\rho_0}} \frac{\triangle \phi}{\phi} u^2\, dV_g - 2 \int_{X \setminus \bD_{\rho_0}} u \left\<d\phi, d\frac{u}{\phi}\right\>\, dV_g
      - \int_{X \setminus \bD_{\rho_0}} u \phi \triangle \frac{u}{\phi}\, dV_g\\
  & \geq  \left(\frac{(n-1)^2}{4} - \epsilon\right) \int_{X \setminus \bD_{\rho_0}} u^2\, dV_g - 2 \int_{X \setminus \bD_{\rho_0}} u \left\<d\phi, d\frac{u}{\phi}\right\>\, dV_g\\
  & \qquad + \int_{X \setminus \bD_{\rho_0}} \left\< d(u \phi), d\frac{u}{\phi}\right\>\, dV_g\\
  & \geq  \left(\frac{(n-1)^2}{4} - \epsilon\right) \int_{X \setminus \bD_{\rho_0}} u^2\, dV_g - \int_{X \setminus \bD_{\rho_0}} u \left\<d\phi, d\frac{u}{\phi}\right\>\, dV_g\\
  & \qquad + \int_{X \setminus \bD_{\rho_0}} \phi\left\< d u, d\frac{u}{\phi}\right\>\, dV_g\\
  & \geq \left(\frac{(n-1)^2}{4} - \epsilon\right) \int_{X \setminus \bD_{\rho_0}} u^2\, dV_g + \int_{X \setminus \bD_{\rho_0}} \phi^2 \left|d\frac{u}{\phi}\right|^2\, dV_g\\
  & \geq \left(\frac{(n-1)^2}{4} - \epsilon\right) \int_{X \setminus \bD_{\rho_0}} u^2\, dV_g.
\end{aligned}
\]
\end{proof}

As noted in \cite{GL} and \cite{Le}, this simple estimate immediately yields an estimate for the
covariant Laplacian acting on tensor fields by making use of Kato's inequality. Unfortunately,
this estimate is not sharp enough to get useful estimates. In \cite{Le}, Lee mainly deals with
symmetric 2-tensors. In order to get sharp estimates he considers $r$-tensor fields as
$(r-1)$-tensor-valued $1$-forms. However our interest is in tensors which can be seen as
$\Lambda^2 X$-valued $2$-forms. Some new observations are needed. Let us begin with the following
definition:

\begin{defn}\label{T2Xvaluedpform}
We say that a $(0, p+2)$-tensor $\omega$ belongs to $\Lambda^p T^{*2} X$, if it satisfies
\begin{eqnarray*}
\lefteqn{\omega(Y_1,Y_2; Z_1,\cdots,Z_s,\cdots,Z_l,\cdots,Z_p)}\\
  & = & -\omega(Y_1,Y_2; Z_1,\cdots,Z_{s-1},Z_l,Z_{s+1},\cdots,Z_{l-1},Z_s, Z_{l+1},\cdots,Z_p),
\end{eqnarray*}
for every $Y_1,Y_2,Z_1,\cdots,Z_p \in TX$ and any pair $s, l$ with $1 \leq s < l \leq p$.
\end{defn}

It can be easily shown that in local coordinates $(x^\mu)$ a $(0, p+2)-tensor~\omega\in \Lambda^p T^{*2} X$ can be written as
$$\omega=\frac{1}{p!}\omega_{\mu\nu \alpha_1\cdots \alpha_p}dx^\mu\otimes dx^\nu\otimes(dx^{\alpha_1}\wedge\cdots \wedge dx^{\alpha_p}),$$
where the coefficients
$$
\omega_{\mu\nu\alpha_1\cdots \alpha_p}=\omega(\frac{\partial}{\partial x^\mu},\frac{\partial}{\partial x^\nu};
  \frac{\partial}{\partial x^{\alpha_1}},\cdots,\frac{\partial}{\partial x^{\alpha_p}})
$$
satisfy

$$\omega_{\mu\nu \alpha_1\cdots \alpha_l\cdots \alpha_s\cdots \alpha_p}=-\omega_{\mu\nu \alpha_1\cdots \alpha_s\cdots \alpha_l\cdots \alpha_p}\quad (1\leq s< l\leq p).$$

For any local orthogonal frame $\{e_\mu\}$ and dual coframe $\{e^\mu\}$, the exterior derivative
$$D:C^\infty(X;\Lambda^p T^{*2} X )\rightarrow C^\infty(X;\Lambda^{p+1} T^{*2} X)$$
on $T^{*2} X$-valued $p$-forms is given by
$$D\omega \definedas e^\mu\wedge\nabla_{e_\mu}\omega$$
for every $\omega\in\Lambda^p T^{*2} X$. It is standard matter to check that $D$ does not depend on
the choice of the frame $\{e^\mu\}$, see e.g. \cite{BishopGoldberg}. This can be seen as a consequence
of the following proposition which gives an intrinsic definition of $D$:

\begin{prop}\label{D}
If $\omega\in\Lambda^{p} T^{*2} X$, then
 $$D\omega (X_1,X_2; Y_0,\cdots, Y_p)=\sum_{m=0}^{p}(-1)^m(\nabla_{Y_m}\omega)(X_1,X_2; Y_0,\cdots,\widehat{Y}_m, \cdots Y_{p}),$$
 for any $X_1,X_2,Y_0,\cdots, Y_{p}\in TX.$
\end{prop}

\begin{proof}
Choose a point $x \in X$ and an orthonormal frame $\{e_\mu\}$ such that $\nabla e_\mu = \nabla e^\mu = 0$ at $x$,
where $\{e^\mu\}$ is the coframe dual to $\{e_\mu\}$. For a $\omega\in\Lambda^p T^{*2} X$, $\omega$ can be written as

$$\omega=\frac{1}{p!}\omega_{\mu\nu; \alpha_1\alpha_2\cdots\alpha_p} e^\mu \otimes e^\nu \otimes(e^{\alpha_1}\wedge e^{\alpha_2}\wedge\cdots e^{\alpha_p}).$$

Then computing at $x$,

\begin{eqnarray*}
D\omega
 & = & e^\sigma\wedge\nabla_{e_\sigma}\omega\\
 & = & \frac{1}{p!}e^\sigma\wedge\nabla_{e_\sigma}(\omega_{\mu\nu; \alpha_1\alpha_2\cdots \alpha_p}e^\mu\otimes e^\nu\otimes(e^{\alpha_1}\wedge e^{\alpha_2}\wedge\cdots e^{\alpha_p}))\\
 & = & \frac{1}{p!}e^\sigma\wedge(\nabla_{e_\sigma}\omega_{\mu\nu; \alpha_1\alpha_2\cdots \alpha_p})e^\mu\otimes e^\nu\otimes(e^{\alpha_1}\wedge e^{\alpha_2}\wedge\cdots e^{\alpha_p})\\
 & = & \frac{1}{p!}(\nabla_{e_\sigma}\omega_{\mu\nu; \alpha_1\alpha_2\cdots \alpha_p})e^\mu\otimes e^\nu\otimes(e^\sigma\wedge e^{\alpha_1}\wedge e^{\alpha_2}\wedge\cdots e^{\alpha_p}).\\
\end{eqnarray*}

Hence

\begin{eqnarray*}
D\omega (Y_0,\cdots, Y_p)
 & = & \frac{1}{p!}(\nabla_{e_\sigma}\omega_{\mu\nu; \alpha_1\alpha_2\cdots \alpha_p})e^\mu\otimes e^\nu\otimes(e^\sigma\wedge e^{\alpha_1}\wedge\cdots \wedge e^{\alpha_p})(Y_0,\cdots, Y_p)\\
 & = & \left(\frac{1}{p!}\sum_{m=0}^{p}(-1)^m e^\sigma(Y_m)(\nabla_{e_\sigma}\omega_{\mu\nu; \alpha_1\alpha_2\cdots \alpha_p})e^\mu\otimes e^\nu\otimes( e^{\alpha_1}\wedge\cdots \wedge e^{\alpha_p})\right)\\
 & & \qquad \cdot (Y_0,\cdots,\widehat{Y}_m, \cdots Y_p)\\
 & = & \sum_{m = 0}^{p}(-1)^m e^\sigma(Y_m)(\nabla_{e_\sigma}\omega)(Y_0,\cdots,\widehat{Y}_m, \cdots Y_p)\\
 & = & \sum_{m = 0}^{p}(-1)^m (\nabla_{Y_m}\omega)(Y_0,\cdots,\widehat{Y}_m, \cdots Y_p).
\end{eqnarray*}
\end{proof}

Let $D^\ast$ be the formal $L^2$-adjoint of $D$. If $\omega\in\Lambda^p T^{*2} X$, we define the divergence of $\omega$, $\divg ~\omega \in\Lambda^{p-1} T^{*2} X$, as follows:

$$\divg~\omega(X_1, X_2; Y_1,\cdots, Y_{p-1}) \definedas \sum_{m=1}^{n}(\nabla_{e_m}\omega)(X_1, X_2; e_m,Y_1,\cdots, Y_{p-1}).$$

In local coordinates, that is

$$(\divg~\omega)_{\mu\nu; \alpha_1\alpha_2\cdots \alpha_{p-1}} = g^{\gamma\delta}\nabla_\gamma \omega_{\mu\nu; \delta\alpha_1\cdots \alpha_{p-1}}.$$

\begin{prop}\label{divergence}
On $\Lambda^p T^{*2} X$, $D^\ast=-\divg.$
\end{prop}

\begin{proof}
Select arbitrary $\theta\in\Lambda^p T^{*2} X$ and $\omega\in\Lambda^{p-1} T^{*2} X$ with compact support.
Then it follows from Proposition \ref{D} that

\begin{eqnarray*}
\int_{X^n}\left\<\theta, D\omega\right\> dV_g
  & = & \frac{1}{p!}\int_{X^n} \theta^{\mu\nu; \alpha_0\cdots \alpha_{p-1}} \left(\sum_{m=0}^{p-1}(-1)^m\nabla_{e_{\alpha_\sigma}}\omega_{\mu\nu; \alpha_0\cdots \widehat{\alpha}_m \cdots \alpha_{p-1}}\right) dV_g\\
  & = & \frac{1}{p!}\int_{X^n} \sum_{m=0}^{p-1} \theta^{\mu\nu; \alpha_m\alpha_0\cdots\widehat{\alpha}_m\cdots \alpha_{p-1}} \nabla_{e_{k_m}}\omega_{\mu\nu; \alpha_0\cdots \widehat{\alpha}_m\cdots \alpha_{p-1}} dV_g\\
  & = & \frac{1}{(p-1)!}\int_{X^n} \sum_{m=0}^{p-1} \theta^{\mu\nu; \alpha_m \alpha_0\cdots\widehat{\alpha}_m\cdots \alpha_{p-1}} \nabla_{e_{\alpha_m}}\omega_{\mu\nu; \alpha_0\cdots \widehat{\alpha}_m\cdots \alpha_{p-1}} dV_g\\
  & = & \frac{1}{(p-1)!}\int_{X^n} \sum_{m=0}^{p-1} (-\divg~\theta)^{\mu\nu; \alpha_0\cdots\widehat{\alpha}_m\cdots \alpha_{p-1}} \omega_{\mu\nu; \alpha_0\cdots \widehat{\alpha}_m\cdots \alpha_{p-1}} dV_g\\
  & = & \int_{X^n}\left\< -\divg~\theta, \omega \right\> dV_g.
\end{eqnarray*}
\end{proof}

We define the Hodge Laplacian on $T^{*2} X$-valued $p$-forms $\Lambda^p T^{*2} X$ as follows

$$\widetilde{\triangle} \definedas DD^*+D^*D,$$

\noindent and the covariant Laplace operator on $\omega \in \Lambda^p T^{*2} X$ by

$$\triangle \omega = \tr(\nabla^2 \omega),$$
where the trace is taken with respect to the two indices of the Hessian.

\begin{prop}
If $\omega\in T^{*2} X$, then $\widetilde{\triangle}\omega=-\triangle \omega$.
\end{prop}

For a 1-form $\theta \in T^*X$, we let $\theta\vee: \Lambda^p T^{*2} X \rightarrow \Lambda^{p-1} T^{*2} X$
denote the adjoint of the map $\theta\wedge: \Lambda^{p-1} T^{*2} X \rightarrow \Lambda^p T^{*2} X$ with respect to $g$,
so that $\left\<\theta\wedge \omega, \eta\right\> = \left\<\omega, \theta\vee\eta\right\>$ for
$\omega\in\Lambda^p T^{*2} X$ and $\eta\in\Lambda^{p+1} T^{*2} X$. In coordinates,

$$(\theta\vee \omega)_{\mu\nu; \alpha_1\cdots \alpha_{p-1}} = g^{\gamma\delta}\theta_\gamma \omega_{\mu\nu; \delta \alpha_1\cdots \alpha_{p-1}}.$$

For any $\xi \in \Lambda^p T^{*2} X$ and any function $u$, we define $H(u)\xi$ as

\begin{equation} \label{eqDefH}
 H(u)\xi \definedas (\hessdd{e_\mu}{e_\nu} u) e^\mu\wedge(e^\nu\vee\xi).
\end{equation}

\begin{prop}\label{propH}
Let $\omega \in \Lambda^p T^2X$ and $f$ be a function. We have
\begin{enumerate}
\item $D(f\omega) = f D \omega+ df \wedge \omega;$
\item $D^*(f\omega) = f D^*\omega - df \vee \omega;$
\item $D^*(df\wedge\omega) = - (\triangle f) \omega - \nabla_{\grad f} \omega - df \wedge D^*\omega + H(f)\omega;$
\item $|df\wedge\omega|^2 + |df\vee\omega|^2 = |df|^2 |\omega|^2.$
\end{enumerate}
\end{prop}

\begin{proof}~
\begin{enumerate}
\item According to the definition,
\begin{eqnarray*}
D(f\omega)
 & = & e^\mu \wedge\nabla_{e_\mu}(f\omega)\\
 & = & e^\mu \wedge(e_\mu(f) \omega + f\nabla_{e^\mu}\omega)\\
 & = & f D\omega + df\wedge\omega;
\end{eqnarray*}

\item In local coordinates,
\begin{eqnarray*}
D^*(f\omega)_{\mu\nu; \alpha_1\cdots \alpha_{p-1}}
 & = & -\divg(f\omega)_{\mu\nu; \alpha_1\cdots \alpha_{p-1}}\\
 & = & -g^{\gamma\delta}\nabla_\gamma (f\omega)_{\mu\nu; \alpha_1 \cdots \alpha_{p-1}}\\
 & = & -g^{\gamma\delta} (f \nabla_\gamma \omega_{\mu\nu; \delta \alpha_1\cdots \alpha_{p-1}} + \nabla_\gamma f \omega_{\mu\nu; \delta\alpha_1\cdots \alpha_{p-1}})\\
 & = & - f \divg\omega - g^{\gamma\delta} \nabla_a f \omega_{\mu\nu; \delta\alpha_1\cdots \alpha_{p-1}}\\
 & = & f D^*\omega_{\mu\nu; \alpha_1\cdots \alpha_{p-1}} - (df \vee \omega)_{\mu\nu; \alpha_1\cdots \alpha_{p-1}};
\end{eqnarray*}

\item \begin{eqnarray*}
(df \wedge \omega)_{\mu\nu; \delta\alpha_1\cdots \alpha_p}
  & = & (\nabla_b f) \omega_{\mu\nu; \alpha_1\cdots \alpha_p} + \sum_{m=1}^p (-1)^m (\nabla_{\alpha_m}f)\omega_{\mu\nu; \delta\alpha_1\cdots\widehat{\alpha_m}\cdots \alpha_p},\\
D^*(df\wedge \omega)_{\mu\nu; \alpha_1\cdots \alpha_p}
 & = & - \divg(df\wedge \omega)_{\mu\nu; \alpha_1\cdots \alpha_p}\\
 & = & -g^{\gamma\delta} \nabla_\gamma (df\wedge \omega)_{\mu\nu; \delta\alpha_1\cdots \alpha_p}\\
 & = & -g^{\gamma\delta} (\hessdd{\gamma}{\delta} f \omega_{\mu\nu; \alpha_1\cdots \alpha_p} + (\nabla_\delta f) \nabla_\gamma \omega_{\mu\nu; \alpha_1\cdots \alpha_p})\\
 & & \qquad -\sum_{m=1}^p g^{\gamma\delta}((-1)^m (\hessdd{\alpha_m}{\gamma} f) \omega_{\mu\nu; \delta\alpha_1\cdots\widehat{\alpha_m}\cdots \alpha_p}\\
 & & \qquad + (-1)^m (\nabla_{\alpha_m} f) \nabla_\gamma \omega_{\mu\nu; \delta\alpha_1\cdots\widehat{\alpha_m}\cdots \alpha_p})\\
 & = & ((-\triangle f) \omega-\nabla_{\grad f} \omega+ H(f)\omega-df\wedge D^*\omega)_{\mu\nu; \alpha_1\cdots \alpha_p}.
\end{eqnarray*}

\item \begin{eqnarray*}
|df\wedge\omega|^2
 & = &\frac{1}{(p+1)!}\left(\sum_{m=1}^p(-1)^m (\nabla_{\alpha_m} f) \omega_{\mu\nu; \alpha_0\cdots\widehat{\alpha_m}\cdots \alpha_p}\right) \left(\sum_{s=1}^p(-1)^s (\nabla^{\alpha_s} f) \omega^{\mu\nu; \alpha_0\cdots\widehat{\alpha_s}\cdots \alpha_p}\right)\\
 & = & \frac{1}{(p+1)!}\sum_{m=1}^p\left((-1)^m (\nabla_{\alpha_m} f) \omega_{\mu\nu; \alpha_0\cdots\widehat{\alpha_{m}}\cdots \alpha_p} \sum_{s=1,s \neq m}^p (-1)^s (\nabla^{\alpha_s} f) \omega^{\mu\nu; \alpha_0\cdots\widehat{\alpha_s}\cdots \alpha_p}\right.\\
 && \qquad \left. + (\nabla_{\alpha_m} f) \omega_{\mu\nu; \alpha_0\cdots\widehat{\alpha_{m}}\cdots \alpha_p} (\nabla^{\alpha_m} f) \omega^{\mu\nu; \alpha_0\cdots\widehat{\alpha_{m}}\cdots \alpha_p}\right)\\
 & = & |df|^2 |\omega|^2 + \frac{1}{(p+1)!}\sum_{m=1}^p \left((-1)^m (\nabla_{\alpha_m} f) \omega_{\mu\nu; \alpha_0\cdots\widehat{\alpha_{m}}\cdots \alpha_p} \sum_{s=1,s \neq m}^p (-1)^s (\nabla^{\alpha_s} f) \omega^{\mu\nu; \alpha_0\cdots\widehat{\alpha_s}\cdots \alpha_p} \right)\\
 & = & |df|^2 |\omega|^2 - \frac{1}{(p+1)!}\sum_{m=1}^p\sum_{s=1,s\:\neq m}^p (\nabla_{\alpha_m} f) \omega_{\mu\nu; \alpha_s\alpha_0\cdots\widehat{\alpha_{m}}\cdots \widehat{\alpha_{s}}\cdots \alpha_p} (\nabla^{\alpha_s} f) \omega_{\mu\nu; \alpha_m\alpha_0\cdots\widehat{\alpha_{m}}\cdots\widehat{\alpha_{s}}\cdots \alpha_p}\\
 & = & |df|^2 |\omega|^2 - |df\vee\omega|^2.
\end{eqnarray*}
\end{enumerate}
\end{proof}

The following lemma is taken from \cite[Lemma 7.9]{Le}:

\begin{lem}\label{lmWitten}
For any smooth compactly supported section $\xi$ of $\Lambda^q T^{*2}X$, and any positive $C^2$ function $\varphi$ on $X$, the following integral formula holds
$$(\xi,\widetilde{\triangle}\xi)\geq\int_X\left<\xi, (-\varphi^{-1}\triangle\varphi+2H(\log \varphi)\xi)\right> dV_g.$$
Here $\left<\cdot,\cdot\right>$ is the induced inner product of tensor bundles and $(\cdot,\cdot)$ is $\int_{X^n} \left<\cdot,\cdot\right>dV_g$.
\end{lem}

As in \cite[Lemma 7.10]{Le} and \cite[Lemma 2.2]{LY}, we also have the following result:

\begin{lem}\label{q}
Let $(X^n,g)$ be a complete non-compact Einstein manifold of dimension $n \geq 6$. Then for every
small $\epsilon > 0$ there exists a compact set $K_1(\epsilon)$ such that the following
estimate holds for any smooth section $\xi$ of $\Lambda^2 T^{*2} X$ compactly supported in
$X^n \setminus K_1(\epsilon)$:
\[
(\xi,\widetilde{\triangle}\xi)\geq \left[\frac{(n-5)^2}{4}-C(n, \epsilon)\right]\int_{X^n} |\xi|^2 dV_g.
\]
\end{lem}

\begin{proof}
We let $\{e^\mu\}$, $0\leq \mu \leq n$ be a local orthonormal coframe of $g$ such that $e_0 = d\rho$. This implies
that $\{e_i\}$, $1 \leq i \leq n-1$ is tangent to $\Sigma_\rho$. For convenience, we also denote $g = d\rho^2 +
g_{ij}(\rho, \theta) dx^i dx^j$. We denote $\epsilon' \definedas \frac{\epsilon}{3n-11}$. We set

$$\varphi_2(x)= e^{-\frac{n-5}{2}\rho}.$$

Arguing as in the proof of Lemma \ref{lmMeanCurvEstimate}, we get that if $\rho_0$ is large enough,

\begin{equation*}
|\sff - \delta| \leq \epsilon',
\end{equation*}
on $X \setminus \bD_{\rho_0}$. Restricting ourselves to $X \setminus \bD_{\rho_0}$, this implies that 

\begin{eqnarray}
-\varphi_2^{-1}\triangle\varphi_2
  & =    & - \frac{(n-5)^2}{4} + \frac{n-5}{2}\triangle \rho\nonumber\\
  & \geq & - \frac{(n-5)^2}{4} + \frac{(n-5)(n-1)}{2} - (n-1) \epsilon\nonumber\\
  & =    & \frac{(n-5)(n+3)}{4} - (n-1)\epsilon',\label{triangle}
\end{eqnarray}
and
\begin{eqnarray*}
\hessdd{i}{j} \log\varphi_2
  & = &    -\frac{n-5}{2} \hessdd{i}{j}\rho\\
  & \geq & -\frac{n-5}{2} g_{ij} - \epsilon' g_{ij},\\
\hessdd{0}{j} \log\varphi_2
  & = & 0,\\
\hessdd{0}{0} (\log\varphi_2)\\
  & = & 0.
\end{eqnarray*}

From these estimates, we get that, for any $\xi$ which is compactly supported in $X \setminus \bD_{\rho_0}$,

\begin{eqnarray}
\left\<2H(\log\varphi_2)\xi,\xi\right\>
  & =    & \left\<2 \hessdd{\mu}{\nu}(\log\varphi_2) e^\mu \wedge (e^\nu \vee \xi), \xi\right\>\nonumber\\
  & =    & 2 \hessdd{\mu}{\nu}(\log\varphi_2) \left\< e^\mu \vee \xi, e^\nu \vee \xi\right\>\nonumber\\
  & =    & 2 \hessdd{e^i}{e^j}(\log\varphi_2) \left\< e^i \vee \xi, e^j \vee \xi\right\>\nonumber\\
  & \geq & - 2 \left(\frac{n-5}{2} + \frac{n-5}{2} \epsilon' \right)\delta_{ij} \left\< e^i \vee \xi,e^j \vee \xi\right\>\nonumber\\
  & \geq & - (2(n-5)+ (n-5)\epsilon) \left\<\xi,\xi\right\>\label{H}.
\end{eqnarray}

Here we have used the fact that
\[\sum_{i=1}^{n-1}|e^i \vee \xi|^2 \leq \sum_{\mu=0}^{n-1}\left|e^\mu \vee \xi\right|^2 = 2 |\xi|^2,\]
for $\xi\in\Lambda^2 T^2X$. Combining equation \eqref{H} and \eqref{triangle} and Lemma \ref{lmWitten}, we have

\begin{eqnarray*}
\left(\xi,\widetilde{\triangle}\xi\right)
  & \geq & \int_{X^n} \left(\frac{(n-5)(n+3)}{4}-2(n-5)-(3n-11)\epsilon'\right)\left\<\xi,\xi\right\> dV_g\\
  & \geq & \left(\frac{(n-5)^2}{4}- \epsilon\right) \int_{X^n} |\xi|^2 dV_g.
\end{eqnarray*}

This proves the lemma with $K_1(\epsilon) = \bD_{\rho_0}$.
\end{proof}

Note that a (0,4)-tensor $\omega$ such that $\omega(\cdot,\cdot; Y_1, Y_2) = -\omega(\cdot,\cdot; Y_2, Y_1)$
for any $Y_1, Y_2\in T X$,  can be considered as a $\Lambda^2 X$-valued 2-form, i.e.
$\omega \in \Lambda^2(X, \Lambda^2 X)$. In the remaining of this section we will consider such (0,4)-tensors.
The following lemma gives a Weitzenb\"ock formula relating the covariant Laplacian on such tensors
to $\widetilde{\triangle}$:

\begin{lem}\label{f}
For a section $\omega$ of $\Lambda^2 (X, \Lambda^2 X)$,

\begin{eqnarray*}
\widetilde{\triangle}\omega_{\alpha\beta\gamma\delta}
 & = &- \triangle \omega_{\alpha\beta\gamma\delta} + \ricud{\nu}{\gamma} \omega_{\alpha\beta \nu\delta} - \ricud{\nu}{\delta} \omega_{\alpha\beta \nu\gamma}\\
 & & \qquad - \riemdudu{\delta}{\nu}{\gamma}{\mu} \omega_{\alpha\beta \mu\nu} + \riemuddu{\nu}{\alpha}{\gamma}{\mu} \omega_{\nu\beta \mu\delta} + \riemuddu{\nu}{\beta}{\gamma}{\mu} \omega_{\alpha\nu \mu\delta}\\
 & & \qquad - \riemduud{\gamma}{\nu}{\mu}{\delta} \omega_{\alpha\beta \mu\nu} + \riemudud{\nu}{\alpha}{\mu}{\delta} \omega_{\nu\beta \mu\gamma} + \riemudud{\nu}{\beta}{\mu}{\delta} \omega_{\alpha\nu \mu\gamma}.
\end{eqnarray*}
\end{lem}

\begin{proof}
Note that the last two indices of $\omega$ are considered to be the $2$-form indices. By Proposition \ref{divergence}, Proposition \ref{D}
and some direct computations, we get

\begin{eqnarray*}
(D^*D\omega)_{\alpha\beta\gamma\delta}
  & = & - \nabla^\mu(D\omega)_{\alpha\beta\mu\gamma\delta}\\
  & = &-\nabla^\mu \nabla_\mu \omega_{\alpha\beta\gamma\delta} + \nabla^\mu \nabla_\gamma \omega_{\alpha\beta\mu\delta} - \nabla^\mu \nabla_\delta \omega_{\alpha\beta\mu\gamma};\\
(DD^*\omega)_{\alpha\beta\gamma\delta}
  & = & (D^*\omega_{\alpha\beta\delta})_\gamma-(D^*\omega_{\alpha\beta\gamma})_\delta\\
  & = & -\nabla_\gamma\nabla^\mu \omega_{\alpha\beta\mu\delta} + \nabla_\delta \nabla^\mu \omega_{\alpha\beta\mu\gamma}.
\end{eqnarray*}

\noindent Then, applying the Ricci identity,

$$
\nabla_\delta \nabla_\gamma \omega_{\alpha_1\cdots \alpha_4} - \nabla_\gamma \nabla_\delta \omega_{\alpha_1\cdots \alpha_4}
 = \sum_{s=1}^4\omega_{\alpha_1\cdots \alpha_{s-1} \nu \alpha_{s+1} \cdots \alpha_4} \riemuddd{\nu}{\alpha_s}{\gamma}{\delta},$$

\noindent we finally get

\begin{eqnarray*}
\widetilde{\triangle}\omega_{\alpha\beta\gamma\delta}
  & = & (DD^*+D^*D)\omega_{\alpha\beta\gamma\delta}\\
  & = &  -\nabla^\mu \nabla_\mu \omega_{\alpha\beta\gamma\sigma} + \nabla^\mu \nabla_\gamma \omega_{\alpha\beta\mu\sigma} - \nabla_\gamma \nabla^\mu \omega_{\alpha\beta\mu\sigma},^\mu     
      +  \nabla_\sigma \nabla^\mu \omega_{\alpha\beta\mu\gamma} - \nabla^\mu \nabla_\sigma \omega_{\alpha\beta\mu\gamma}\\
  & = &  - \triangle \omega_{\alpha\beta\gamma\sigma} + \ricud{\nu}{\gamma} \omega_{\alpha\beta\nu\sigma} + \riemuddu{\nu}{\delta}{\gamma}{\mu} \omega_{\alpha\beta\mu\nu} - \riemdudu{\alpha}{\nu}{\gamma}{\mu} \omega_{\nu\beta\mu\sigma} - \riemdudu{\beta}{\nu}{\gamma}{\mu} \omega_{\alpha\nu\mu\sigma}\\
  & & \qquad - \ricud{\nu}{\delta} \omega_{\alpha\beta\nu\gamma} + \riemudud{\nu}{\gamma}{\mu}{\delta} \omega_{\alpha\beta\mu\nu} - \riemduud{\alpha}{\nu}{\mu}{\delta} \omega_{\nu\beta\mu\gamma} - \riemduud{\beta}{\nu}{\mu}{\delta} \omega_{\alpha\nu\mu\gamma}.
\end{eqnarray*}
\end{proof}

\begin{defn}\label{defSigma}
We say that a $4$-tensor $\omega$ belongs to $\widetilde{\Sigma}^4$ if it satisfies the following three assumptions:
\begin{enumerate}
 \item $\omega_{\alpha\beta\gamma\delta} = - \omega_{\beta\alpha\gamma\delta}$,
 \item $\omega_{\alpha\beta\gamma\delta} + \omega_{\alpha\gamma\delta\beta} + \omega_{\alpha\delta\beta\gamma} = 0,$
 \item $\omega_{\alpha\beta\gamma\delta} = \omega_{\gamma\delta\alpha\beta}.$
\end{enumerate}
Furthermore, if $\omega$ is trace-free, meaning that $g^{ik}\omega_{ijkl}=0$, we say that $\omega \in \widetilde{\Sigma}_0^4$. 
\end{defn}

Note that any element of $\Sigmatil^4$ belongs to $\Lambda^2(X, \Lambda^2 X)$. Combining Lemmas \ref{q} and \ref{f}, we obtain the following estimate:

\begin{lem}\label{spetrum}
Let $(X^n,g)$ be an $n$-dimensional Einstein manifold containing an essential set $\bD$ with $n > 5$.
Then for every $\epsilon > 0$ there exists a compact set $K_2(\epsilon) \supset \bD$ such that the
following estimate holds for any smooth $4$-tensor $\omega \in \widetilde{\Sigma}_0^4$ compactly supported
in $X^n \setminus K_2(\epsilon)$:
\[
\int_{X^n} \left|\nabla \omega\right|^2 dV_g\geq \left(\frac{(n-1)^2}{4}+4-C(n, \epsilon)\right)\int_{X^n} \left|\omega\right|^2 dV_g.
\]
\end{lem}

\begin{proof}
From Lemma \ref{lmDecayWeyl}, there exists a compact set $K_2(\epsilon) \supset \bD$ such that
$$\|\riem- \bfK \|_{L^\infty(X^n \setminus K_2(\epsilon))} = \left\|\weyl\right\|_{L^\infty(X^n \setminus K_2(\epsilon))}  \leq \epsilon.$$

By a direct computation, we have

\begin{eqnarray*}
\ricud{\nu}{\gamma} \omega_{\alpha\beta\nu\delta}- \ricud{\nu}{\delta} \omega_{\alpha\beta\nu\gamma}
 & = & -2(n-1)\omega_{\alpha\beta\gamma\delta};\\
- \riemdudu{\delta}{\nu}{\gamma}{\mu} \omega_{\alpha\beta\mu\nu} + \riemuddu{\nu}{\alpha}{\gamma}{\mu} \omega_{\nu\beta\mu\delta} + \riemuddu{\nu}{\beta}{\gamma}{\mu} \omega_{\alpha\nu\mu\delta}
 & = & -\omega_{\alpha\beta\delta\gamma}-\omega_{\gamma\beta\alpha\delta}-\omega_{\alpha\gamma\beta\delta} + O(\epsilon \omega)\\
 & = & O(\epsilon \omega);\\
- \riemduud{\gamma}{\nu}{\mu}{\delta} \omega_{\alpha\beta\mu\nu} + \riemudud{\nu}{\alpha}{\mu}{\delta} \omega_{\nu\beta\mu\gamma} + \riemudud{\nu}{\beta}{\mu}{\delta} \omega_{\alpha\nu\mu\gamma}
 & = & \omega_{\alpha\beta\gamma\delta} + \omega_{\delta\beta\alpha\gamma} + \omega_{\alpha\delta\beta\gamma} + O(\epsilon \omega)\\
 & = & O(\epsilon \omega).
\end{eqnarray*}

Using Lemma \ref{f} together with Lemma \ref{q}, we get:

\begin{eqnarray*}
\int_{X^n} \left|\nabla \omega\right|^2 dV_g
 &  =  & (\omega, -\triangle \omega)\\
 &  =  & (\omega, \widetilde{\triangle}\omega_{\alpha\beta\gamma\delta} - \ricud{\nu}{\gamma} \omega_{\alpha\beta\nu\delta} + \ricud{\nu}{\delta} \omega_{\alpha\beta\nu\gamma}
    + \riemdudu{\delta}{\nu}{\gamma}{\mu} \omega_{\alpha\beta\mu\nu} - \riemuddu{\nu}{\alpha}{\gamma}{\mu} \omega_{\nu\beta\mu\delta}\\
 &  & \qquad - \riemuddu{\nu}{\beta}{\gamma}{\mu} \omega_{\alpha\nu\mu\delta} + \riemduud{\gamma}{\nu}{\mu}{\delta} \omega_{\alpha\beta\mu\nu} - \riemudud{\nu}{\alpha}{\mu}{\delta} \omega_{\nu\beta\mu\gamma}
    - \riemudud{\nu}{\beta}{\mu}{\delta} \omega_{\alpha\nu\mu\gamma})\\
 & \geq & \frac{(n-5)^2}{4}\int_{X^n} \left|\omega\right|^2 dV_g + 2(n-1) \int_{X^n}\left|\omega\right|^2 dV_g - C(n, \epsilon) \int_{X^n}\left|\omega\right|^2 dV_g\\
 & \geq & \left(\frac{(n-1)^2}{4}+4-C(n, \epsilon)\right)\int_{X^n} \left|\omega\right|^2 dV_g.
\end{eqnarray*}
\end{proof}

\begin{rem}
By a density argument, it is not difficult to see that Lemmas \ref{lmChengYau} and \ref{spetrum}
are still true if we replace the condition that $u$ or $\omega$ has compact support by
$u \in W_0^{1, 2}(X^n\setminus K_1)$ (resp. $\omega \in W_0^{1, 2}(X^n\setminus K_2)$). Here the
subscript $0$ means that $u$ (resp. $\omega$) has vanishing trace on $\partial K_1$ (resp. $\partial K_2$).
\end{rem}

Our next goal is to make use of the above estimates to get weighted $L^2$-estimate for the Weyl
tensor. More precisely, we have:

\begin{prop}\label{lp}
Suppose that $(X^n,g)$, $n \geq 4$, is a complete noncompact Einstein manifold  with an essential set $\bD$. If
$\|\weyl\|_{L^p(X^n,g)} < \infty$, with $1 < p < \frac{n-1}{2}$, then
$$\| \weyl\|_{L^2(X^n,g)} < \infty.$$
Furthermore if $n \geq 6$, we have $\| e^{\frac{a}{2}\rho}\weyl\|_{W^{1,2}(X^n,g)} <\infty$ for any $a \in \left[0; n-5\right)$.
\end{prop}

Before giving the proof of this proposition, we need to make a preliminary definition. Formula \eqref{eqDefQ}
together with Equation \eqref{eqDefB} define a quadratic map from $\widetilde{\Sigma}^4$ to itself.
We define the associated symmetric bilinear map as follows:

\[
\cQ(\xi, \omega)_{\alpha\beta\gamma\delta}
  \definedas \tudud{\xi}{\mu}{\alpha}{\nu}{\beta} \omega_{\mu\gamma\nu\delta} + \tudud{\xi}{\mu}{\alpha}{\nu}{\gamma} \omega_{\mu\beta\nu\delta}
   - \tudud{\xi}{\mu}{\beta}{\nu}{\alpha} \omega_{\mu\gamma\nu\delta} - \tudud{\xi}{\mu}{\beta}{\nu}{\delta} \omega_{\mu\alpha\nu\gamma}.
\]

This map enjoys the following nice property:

\begin{claim}\label{clSymmetry}
For every $\omega,~\xi\in\widetilde{\Sigma}^4$, we have
\[\langle \cQ(\omega, \weyl), \xi\rangle = \langle \omega, \cQ(\xi, \weyl)\rangle.\]
Equivalently, the map $\omega \mapsto \cQ(\weyl, \omega)$ is symmetric.
\end{claim}

\begin{proof} The proof is a straightforward calculation:
\begin{eqnarray*}
\langle \cQ(\omega, \weyl), \xi\rangle
  & = & \left(\tudud{\omega}{\mu}{\alpha}{\nu}{\beta} \weyl_{\mu\gamma\nu\delta} + \tudud{\omega}{\mu}{\alpha}{\nu}{\gamma} \weyl_{\mu\beta\nu\delta}
     - \tudud{\omega}{\mu}{\beta}{\nu}{\alpha} \weyl_{\mu\gamma\nu\delta} - \tudud{\omega}{\mu}{\beta}{\nu}{\delta} \weyl_{\mu\alpha\nu\gamma}\right) \xi^{\alpha\beta\gamma\delta}\\
  & = & 2 \left(\tudud{\omega}{\mu}{\alpha}{\nu}{\beta} \weyl_{\mu\gamma\nu\delta} + \tudud{\omega}{\mu}{\alpha}{\nu}{\gamma} \weyl_{\mu\beta\nu\delta}\right) \xi^{\alpha\beta\gamma\delta}\\
  & = & 2 \tudud{\omega}{\mu}{\alpha}{\nu}{\beta} \weyl_{\mu\gamma\nu\delta} \left(\xi^{\alpha\beta\gamma\delta} + \xi^{\alpha\gamma\beta\delta}\right)\\
  & = & -2 \tudud{\omega}{\mu}{\alpha}{\nu}{\beta} \weyl_{\mu\gamma\nu\delta} \left(2 \xi^{\alpha\gamma\delta\beta} + \xi^{\alpha\delta\beta\gamma}\right)\\
  & = & 4 \omega_{\alpha\mu\beta\nu} \xi^{\alpha\gamma\beta\delta} \weyludud{\mu}{\gamma}{\nu}{\delta} - 2 \omega_{\alpha\mu\beta\nu} \xi^{\alpha\delta\beta\gamma} \weyludud{\mu}{\gamma}{\nu}{\delta}.
\end{eqnarray*}
where we used the first Bianchi identity, Property 2 of Definition \ref{defSigma}, to get the fourth line. Under this form,
the claim becomes clear by swapping $\gamma$ (resp. $\delta$) and $\mu$ (resp. $\nu$).
\end{proof}

Having made this definition, we can give a proof of Proposition \ref{lp}:

\begin{proof}[Proof of Proposition \ref{lp}]
We remark that if $n < 6$, $p < \frac{n-1}{2} \leq 2$. Hence, from the fact that $\weyl \in L^\infty$,
which was proven in Lemma \ref{lmDecayWeyl}, we conclude that $\weyl \in L^2$. As a consequence, we
now restrict our attention to the case $n \geq 6$. We also assume that $p > 2$. For an arbitrary
$b \in \bR$ and using H\"older inequality, we get
\begin{equation}\label{eqWeightedL2}
\int_{X^n} e^{-2b\rho} \left|\weyl\right|^2_g dV_g
  \leq \left(\int_{X^n} \left|\weyl\right|^p_g dV_g\right)^{\frac{2}{p}} \left(\int_{X^n} e^{-\frac{2bp}{p-2}\rho} dV_g\right)^{\frac{p-2}{p}}.
\end{equation}
Note that the second integral appearing in the right-hand side can be rewritten as follows:

\begin{equation}\label{eqWeightedL2b}
\left(\int_{X^n} e^{-\frac{2bp}{p-2}\rho} dV_g\right)^{\frac{2}{p}} = \vol(\bD) + \int_0^\infty e^{-\frac{2bp}{p-2}\rho} \left|\Sigma_\rho\right| d\rho,
\end{equation}
where $\left|\Sigma_\rho\right|$ denotes the area of $\Sigma_\rho$. Using Lemma \ref{lmMeanCurvEstimate}, we get:

\[
 \diff{~}{\rho} \left|\Sigma_\rho\right| = \int_{\Sigma_\rho} H dV_g = (n-1 + o(1)) \left|\Sigma_\rho\right|.
\]
Integrating this differential estimate, we obtain

\[
 \left|\Sigma_\rho\right| = \left|\Sigma_0\right| e^{(n-1)\rho + o(\rho)}.
\]
In particular, Integral \eqref{eqWeightedL2b} converges if and only if $n-1-\frac{2bp}{p-2} < 0$,
that is to say $b > \frac{n-1}{2} - \frac{n-1}{p}$. From Equation \eqref{eqWeightedL2}, we conclude that
\begin{equation}\label{eqWeightedL2p}
 \int_{X^n} e^{-2b\rho} \left|\weyl\right|^2_g dV_g < \infty
\end{equation}
for any $b > \frac{n-1}{2} - \frac{n-1}{p}$.

We now select $\epsilon > 0$ to be fixed later and a cutoff function $\chi$ which vanishes on $K_2 = K_2(\epsilon)$
and which equals one outside a larger compact subset $K_2' \supset K_2(\epsilon)$. We set $\Wtil \definedas \chi W$.
We remark that $\Wtil$ satisfies the following equation:

\begin{equation}\label{eqWtil}
 \triangle \Wtil + 2(n-1) \Wtil + 2\cQ(\weyl, \Wtil) = \theta,
\end{equation}
where $\theta$ is a tensor belonging to $\Sigmatil^4_0$ and whose support is contained in $\mathrm{supp}(\nabla \chi) \subset K_2' \setminus K_2$.

For any compactly supported Lipschitz function $f$, we have
\[
\begin{aligned}
\int_X \left|\nabla (f \Wtil)\right|^2 dV_g
  & = \int_X f^2 \left|\nabla \Wtil\right|^2 dV_g + 2 \int_X f \left\<\nabla f \otimes \Wtil, \nabla \Wtil\right\> dV_g + \int_X \left|\nabla f\right|^2 \left|\Wtil\right|^2 dV_g\\
  & = \int_X \left\<\nabla (f^2 \Wtil), \Wtil\right\> dV_g + \int_X \left|\nabla f\right|^2 \left|\Wtil\right|^2 dV_g\\
  & = - \int_X f^2 \left\<\Wtil, \triangle \Wtil\right\> dV_g + \int_X \left|\nabla f\right|^2 \left|\Wtil\right|^2 dV_g\\
  & = 2(n-1) \int_X f^2 \left|\Wtil\right|^2 dV_g + 2 \int_X f^2 \left\<\Wtil, \cQ(\weyl, \Wtil)\right\> dV_g\\
  & \qquad - \int_X f^2 \left\<\Wtil, \theta\right\> dV_g + \int_X \left|\nabla f\right|^2 \left|\Wtil\right|^2 dV_g.
\end{aligned}
\]

Since $f \Wtil$ is compactly supported in $X \setminus K_2$, we conclude from Lemma \ref{spetrum} that
\[
\begin{aligned}
\left[\frac{(n-1)^2}{4} + 4 - \epsilon\right] \int_X f^2 \left|\Wtil\right|^2 dV_g
  & \leq 2(n-1) \int_X f^2 \left|\Wtil\right|^2 dV_g + 2 \int_X f^2 \left\<\Wtil, \cQ(\weyl, \Wtil)\right\> dV_g\\
  & \qquad - \int_X f^2 \left\<\Wtil, \theta\right\> dV_g + \int_X \left|\nabla f\right|^2 \left|\Wtil\right|^2 dV_g.
\end{aligned}
\]

From the fact that $\left|\weyl\right| < \epsilon$ on $X^n \setminus K_2$, we conclude that

\begin{equation}\label{eqEstimateWeyl}
\left[\frac{(n-5)^2}{4} - C(n, \epsilon)\right] \int_X f^2 \left|\Wtil\right|^2 dV_g
  \leq - \int_X f^2 \left\<\Wtil, \theta\right\> dV_g + \int_X \left|\nabla f\right|^2 \left|\Wtil\right|^2 dV_g.
\end{equation}

By a simple density argument using Inequality \eqref{eqWeightedL2p}, it can be shown that Estimate \eqref{eqEstimateWeyl}
still holds for any function $f$ such that $f, |\nabla f| = O (e^{-b \rho})$ for some $b > \frac{n-1}{2} - \frac{n-1}{p}$.
We choose $f = f_R(\rho)$ where $f_R$ is a 1-parameter family of functions defined as follows:
\[
f_R(\rho) \definedas
 \left\lbrace\begin{array}{ll}
   e^{a\rho}\qquad & \text{if } \rho \leq R,\\
   e^{aR - b(\rho-R)} & \text{if } \rho \geq R.\end{array}\right.
\]

It is easy to see that these functions are Lipschitz continuous and satisfy $f, |\nabla f| = O (e^{-b \rho})$.
From the fact that $|\nabla f|^2 = a^2 f^2$ if $\rho < R$ and $|\nabla f|^2 = b^2 f^2$ if $\rho > R$, we finally get:

\begin{equation}\label{eqEstimateWeyl2}
\begin{split}
\left[\frac{(n-5)^2}{4} - a^2 - C(n, \epsilon)\right] \int_{\bD_R} f^2 \left|\Wtil\right|^2 dV_g
+ \left[\frac{(n-5)^2}{4} - b^2 - C(n, \epsilon)\right] \int_{X \setminus \bD_R} f^2 \left|\Wtil\right|^2 dV_g\\
  \leq - \int_X f^2 \left\<\Wtil, \theta\right\> dV_g.
\end{split}
\end{equation}

Choosing $b < \frac{n-5}{2}$, which is possible since $\frac{n-1}{2} - \frac{n-1}{p} < \frac{n-5}{2}$, and $\epsilon$
so small that $\frac{(n-5)^2}{4} - b^2 - C(n, \epsilon) \geq 0$, we finally get

\begin{equation}\label{eqEstimateWeyl3}
\left[\frac{(n-5)^2}{4} - a^2 - C(n, \epsilon)\right] \int_{\bD_R} e^{2 a \rho} \left|\Wtil\right|^2 dV_g
\leq - \int_X e^{2a \rho} \left\<\Wtil, \theta\right\> dV_g.
\end{equation}

Letting $R$ tend to infinity, and upon reducing the value of $\epsilon$ so that $\frac{(n-5)^2}{4} - a^2 - C(n, \epsilon) > 0$,
we finally get
\[
 \int_{\bD_R} e^{2 a \rho} \left|\Wtil\right|^2 dV_g < \infty.
\]
This ends the proof of Proposition \ref{lp}.
\end{proof}

%% file: pointwise.tex
In this section, we assume that $(X^n, g)$ is an AHE manifold with Weyl tensor satisfying $\|\weyl\|_{L^p (X,g)}<\infty$
for some $p\leq \frac{n-1}2$. Note that on Einstein manifolds we always have $\weyl= \riem -\bfK$. The main purpose of
this section is to give a pointwise decay estimate for $\weyl$.
We achieve this by two steps: first, we get the estimate by assuming $(X^n,g)$ is a $C^{2,\mu}$-conformally compact
Einstein manifold. Obviously, even in this case the result has its own interests. Later, we remove the condition of
$C^{2,\mu}$-regularity and try to obtain the pointwise estimate of $|\weyl|$ in more general situations. Unfortunately,
due to some technical reasons mentioned in the introduction, we have to assume $n \geq 7$ in this case.

We begin by recaling the definition of a conformally compact manifold:

\begin{defn}
 We say that $(X, g)$ is a $C^{k, \alpha}$-conformally compact manifold if
\begin{itemize}
 \item there exists a smooth manifold $\overline{X}$ with boundary $\partial X$ whose interior is $X$:
$X = \overline{X} \setminus \partial X$
 \item and for some defining function $x$, $\gbar = x^2 g$ extends to a $C^{k, \alpha}$ metric
on $\overline{X}$,
\end{itemize}
where a defining function $x$ is a smooth function $x : \overline{X} \to bR_+$ such that $x^{-1}(0) = \partial X$ with
$dx \neq 0$ at every point of $\partial X$.
\end{defn}
Furthermore, assuming that $\sec_g \to -1$ at infinity, the function $x$ satisfies $|dx|^2_{\gbar} \equiv 1$ on $\partial X$.
$C^{k, \alpha}$-conformally compact manifolds whose curvature tends to $-1$ at infinity are called \emph{asymptotically hyperbolic}.
We refer the reader to \cite{Le} and references therein for more details on these manifolds.

In order to get the pointwise decay of $\weyl$ which is mentioned above, we need the following lemma, which was observed in \cite{W}.

\begin{lem}\label{lmWang}
Suppose that $(X^n,g)$ is a conformally compact Einstein manifold of regularity $C^2$. If its conformal infinity is conformally flat, then
$$|\weyl|=O(r^{n+1})$$
where $r$ is the defining function determined by some conformal infinity.
\end{lem}

Here is the outline of the proof of the above lemma. We refer the reader to \cite{W} for details. Straightforward calculations yield that
if an Einstein metric $g$ is at least $C^2$ conformally compact, then the sectional curvature in $X$ satisfies

\begin{equation}\label{eqDecaySec}
\sec_g = -1 + O(r^2).
\end{equation}

The most basic and important fact about asymptotically hyperbolic manifolds is that a conformal infinity $(\partial X,g_0)$ determines
a unique defining function $r$ in a collar neighborhood of $\partial X$ such that

$$g=r^{-2}(dr^2 + g_r),$$

\noindent where $g_r$ is an $r$-dependent family of metrics on $\partial X$ with $g_r|_{r=0}=g_0$. See e.g. \cite{Le1}.
It follows from the work of Fefferman and Graham \cite{FG} that the Einstein equation implies the following asymptotic
 expansion for the metric $g$. For $n$ even,

$$g_r=g_0 + g_{(2)}r^2 +(\text{even powers}) + g_{(n-2)}r^{n-2} + g_{(n-1)}r^{n-1} +... ,$$

\noindent where the $g_{(j)}$ are tensors on $\partial X$ and $g_{(n-1)}$ is trace-free with respect to $g_0$. The
tensors $g_{(j)}$ for $j\leq n-2$ are locally formally determined by the metric $g_0$, but
$g_{(n-1)}$ is formally undetermined. For $n$ odd the analogous expansion is

$$g_r = g_0 + g_{(2)}r^2 + (\text{even powers}) + kr^{n-1} \log r + g_{(n-1)}r^{n-1} + ... ,$$

\noindent where the $g_{(j)}$'s are locally determined for $j\leq n-2$, $k$ is locally determined and
trace-free, but $g_{(n-1)}$ is formally undetermined.

Due to Theorem A in \cite{CDLS}, we know that if $(X^n,g)$ is a conformally compact of regularity $C^2$ and its conformal infinity
is smooth, then in fact $(X^n,g)$ is conformally compact of order $C^{\infty}$ if $n$ is even or if $n$ is odd and $k \equiv 0$, where $k$
is a conformally covariant tensor. Therefore according Fefferman-Graham expansion, we can get $|\weyl|_g = O(r^{n+1})$ if the conformal
infinity is locally conformally flat.

\begin{thm}\label{curvaturebehavior}
Suppose that $(X^n,g)$ is a conformally compact Einstein manifold of dimension $n \geq 5$ and of
regularity $C^{2, \mu}$ for some $\mu \in (0; 1)$. If we further suppose that there exists
$p \in \left(1; \frac{n-1}{2}\right]$ such that $\|\weyl\|_{L^p(X^n,g)} <\infty$, then
$$|\weyl| = O(r^{n+1}),$$
where $r$ is some special defining function.
\end{thm}

\begin{proof}
Let $r: \overline{X} \to \bR$ be an arbitrary defining function for the conformal infinity $\partial X$
of $X$ and let $\gbar = \rho^2 g$ be the compactified metric. We denote with a bar quantities associated
to the metric $\gbar$, e.g. $\weylbar$ denotes its Weyl tensor. By assumption, $\gbar$ is a $C^{2, \mu}$
 metric on $\overline{X} = X \cup \partial X$.

We first note that $\left|\weyl\right|_g = \rho^2 \left|\weylbar\right|_{\gbar}$. As a consequence,
\[\int_X \left|\weyl\right|_g^p dV_g = \int_X \rho^{2p-n}\left|\weylbar\right|_{\gbar}^p dV_{\gbar}.\]

Since $\gbar$ is $C^{2, \mu}$, $\weylbar$ is a continuous 4-tensor. From the fact that $p \leq \frac{n-1}{2}$,
the function of $\rho$ appearing in the integral on the right-hand side blows up faster than $\rho^{-1}$
when approaching $\partial X$. As a consequence, if this quantity is to be finite, this impose that
$\weylbar \equiv 0$ on $\partial X$.

From the Fefferman-Graham expansion of the metric $g$, we immediately see that the second fundamental form
of $\partial X$ in the manifold $\overline{X}$ vanishes.

If $\ghat$ denotes the metric induced on the conformal infinity $\partial X$, it follows from the Gauss-Codazzi
equations that the Riemann tensor $\hriem$ of $\ghat$ is equal to the restriction of $\riembar$ to $T(\partial X)$.
We denote $\overline{P}$ and $\widehat{P}$ the Schouten tensors of the metrics $\gbar$ and $\ghat$. From the
decomposition of the Riemann tensors
\[
\left\lbrace
\begin{aligned}
 \hriem   & = \ghat \kulk \widehat{P}  + \hweyl,\\
 \riembar & = \gbar \kulk \overline{P} + \weylbar\\
\end{aligned}
\right.
\]
it follows that
\[\hweyl = \weylbar + \left(\overline{P} - \widehat{P}\right)\kulk \ghat,\]
Since $\weylbar \equiv 0$ on $\partial X$, we conclude that
\[\hweyl = \left(\overline{P} - \widehat{P}\right)\kulk \ghat,\]
which implies $\hweyl \equiv 0$ because the two sides of the equality belong to orthogonal subspaces of
$\Sigmatil^4(\partial X)$.

As a consequence, we have proven that $\partial X$ is locally conformally flat. The theorem follows from
Lemma \ref{lmWang}.
\end{proof}

Finally we remove the condition $C^{2,\mu}$-regularity to give the pointwise estimate of $|\weyl|$. According to Proposition \ref{lp} we
get weighted $L^2$-estimate for the Weyl tensor. Using Lemma \ref{lmChengYau}, we are able to show the following theorem:

\begin{thm}\label{mainthm}
Suppose that $(X^n,g)$, $n \geq 7$ is a complete noncompact Einstein manifold with an essential set $\bD$. If
$\| \weyl\|_{L^p(X^n,g)} <\infty$ for some $p \in \left[1; \frac{n-1}{2}\right)$, then
$$|\weyl| \leq Ce^{-(n+1)\rho}.$$
\end{thm}

An essential element in the proof of this theorem is \cite[Theorem 1.2]{HQS} which we recall here for the sake of
completeness:

\begin{proposition}\label{propHQS}
Suppose that $(X^n, g)$ is a complete Riemannian manifold with an essential set $\cD$. If the Riemann tensor
satisfies the following assumptions:
\[
\begin{aligned}
 \left|\riem - \bfK\right|  & = O(e^{-a \rho}),\\
 \left|\nabla \riem \right| & = O(e^{-a\rho})
\end{aligned}
\]
for some constant $a > 2$, then there is a smooth closed manifold $\partial X$ and a smooth structure on $\overline{X} = X \cup \partial X$,
such that setting $x = e^{-\rho}$ and extending it by zero on $\partial X$, $x$ is a defining function for $\partial X$ and the metric
$\gbar =x^2 g$ extends to a $C^{2, \mu}$ metric on the manifold $\overline{X}$ for some $\mu \in (0; 1)$. That is to say $(X, g)$ is
$C^{2, \mu}$-conformally compact.
\end{proposition}

\begin{proof}[Proof of Theorem \ref{mainthm}]
We will assume that $n \geq 8$ and indicate the modifications for $n = 7$.
The first step is to obtain an exponential (pointwise) decay of $|\weyl|$ at infinity. To this end,
we set $\weyl_1 \definedas e^{a\rho} \weyl$ for some $a > 0$ to be chosen later. From Equation
\eqref{equationofweyltensor}, $\weyl_1$ satisfies
\begin{eqnarray}
\triangle \weyl_1
  & = & (\triangle e^{a\rho}) \weyl + 2 \left\<\nabla e^{a\rho}, \nabla \weyl\right\> + e^{a \rho} \triangle \weyl\nonumber\\
  & = & \left[a^2 + (n-1) a + o(1)\right] \weyl_1 + 2 a e^{a\rho} \nabla_{\nabla\rho} \weyl + e^{a \rho} \triangle \weyl\nonumber\\
  & = & \left[a^2 + (n-1) (a-2) + o(1)\right] \weyl_1 + 2 a e^{a\rho} \nabla_{\nabla\rho} \weyl - 2 \cQ(\weyl, \weyl_1),\label{eqWeyl1}
\end{eqnarray}
where we used $\triangle \rho = H = n-1 + o(1)$ (see Lemma \ref{lmMeanCurvEstimate}). Next, we compute $\triangle |\weyl_1|^2$
in two different ways at any point where $|\weyl_1| \neq 0$:

\begin{eqnarray*}
\triangle |\weyl_1|^2
  & = & 2 \left( \left|\nabla |\weyl_1|\right|^2 + |\weyl_1| \triangle |\weyl_1| \right)\\
  & = & 2 \left( \left|\nabla \weyl_1\right|^2 + \left\<\weyl_1, \triangle \weyl_1\right\> \right)\\
  & = & 2 \left( \left|\nabla \weyl_1\right|^2 + \left[a^2 + (n-1) (a-2) + o(1)\right] |\weyl_1|^2\right.\\
  & & \qquad \left. + 2 a e^{a\rho} \left\<\nabla_{\nabla\rho} \weyl, \weyl_1\right\> - 2 \left\<\weyl_1,\cQ(\weyl, \weyl_1)\right\>\right).
\end{eqnarray*}

As a consequence, we get the following equation for $|\weyl_1|$:
\begin{equation}\label{eqNormWeyl1}
|\weyl_1| \triangle |\weyl_1| - \left[a^2 + (n-1) (a-2) + o(1)\right] |\weyl_1|^2
  = \left|\nabla \weyl_1\right|^2 - \left|\nabla |\weyl_1|\right|^2 + 2 a e^{2a\rho} \left\<\nabla_{\nabla\rho} \weyl, \weyl\right\>,
\end{equation}
where we used Lemma \ref{lmDecayWeyl} to get $\<\weyl_1, \cQ(\weyl, \weyl_1)\> = o\left(|\weyl_1|^2\right)$.

The following refined Kato inequality holds for the Weyl tensor of any Einstein manifold (see e.g. \cite{CGH}):
\begin{equation}\label{eqKato}
 \left|\nabla |\weyl|\right| \leq \frac{n-1}{n+1} \left|\nabla \weyl\right|.
\end{equation}

We are going to take advantage of it to estimate the right-hand side of Equation \eqref{eqNormWeyl1}. We first remark that
\begin{eqnarray*}
 \left|\nabla  \weyl_1\right|^2  & = & e^{2a\rho} \left|\nabla  \weyl \right|^2 + 2 a e^{2a\rho} \left\<\nabla_{\nabla\rho}  \weyl,   \weyl \right\> + a^2 e^{2a\rho} |\weyl|^2,\\
 \left|\nabla |\weyl_1|\right|^2 & = & e^{2a\rho} \left|\nabla |\weyl|\right|^2 + 2 a e^{2a\rho} \left\<\nabla_{\nabla\rho} |\weyl|, |\weyl|\right\> + a^2 e^{2a\rho} |\weyl|^2\\
                                 & = & e^{2a\rho} \left|\nabla |\weyl|\right|^2 + a e^{2a\rho} \nabla_{\nabla\rho} |\weyl|^2 + a^2 e^{2a\rho} |\weyl|^2\\
                                 & = & e^{2a\rho} \left|\nabla |\weyl|\right|^2 + 2 a e^{2a\rho} \left\<\nabla_{\nabla\rho}  \weyl,   \weyl \right\> + a^2 e^{2a\rho} |\weyl|^2.
\end{eqnarray*}

Therefore, using Inequality \eqref{eqKato}, we get:
\begin{equation}\label{eqKato2}
\left|\nabla \weyl_1\right|^2 - \left|\nabla |\weyl_1|\right|^2
  \geq \frac{2}{n-1} e^{2a\rho} \left|\nabla  \weyl \right|^2
\end{equation}

Next, using Young's inequality, we remark that
\[
2 a e^{2a\rho} \left\<\nabla_{\nabla\rho} \weyl, \weyl\right\>
  \geq - \frac{2}{n-1} e^{2a\rho} \left|\nabla  \weyl \right|^2 - \frac{n-1}{2} a^2 e^{2a\rho} \left|\weyl\right|^2.
\]

Thus Equation \eqref{eqNormWeyl1} yields the following differential inequality:
\begin{equation}\label{eqNormWeyl3}
 |\weyl_1| \triangle |\weyl_1| - \left[a^2 + (n-1) (a-2) + o(1)\right] |\weyl_1|^2 \geq - \frac{n-1}{2} a^2 \left|\weyl_1\right|^2.
\end{equation}

We select $a = \frac{n-1}{n-3}$. The previous inequality becomes
\begin{equation}\label{eqNormWeyl2}
 \triangle |\weyl_1| \geq \left[- \frac{1}{2} \frac{(n-1)(3n-11)}{n-3} + o(1)\right]\left|\weyl_1\right|,
\end{equation}
at any point where $|\weyl_1| > 0$. Note that when $n > 7$, $a < \frac{n-5}{2}$ so from Proposition \ref{lp}, $W_1 \in L^2$.
We claim that $|W_1| \leq C e^{-\frac{n-1}{2} \rho}$. Indeed, set $b \definedas \frac{n-1}{2} + \delta$ for some small
$\delta > 0$. Then we have
\[
 \left(\triangle + \frac{1}{2} \frac{(n-1)(3n-11)}{n-3}\right) e^{-b\rho}
  = \left(\delta^2-\frac{1}{4} \frac{(n-1)(n-5)^2}{n-3} + o(1)\right) e^{-b\rho}.
\]
Select $\epsilon > 0$ such that \[\epsilon < \frac{1}{4} \frac{(n-1)(n-5)^2}{n-3}.\]
Provided that $\delta^2 < \frac{\epsilon}{2}$, there exists a compact set $K \supset \bD$
such that

\[
\left\lbrace
\begin{aligned}
 \triangle |\weyl_1|  & \geq - \left[\frac{1}{2} \frac{(n-1)(3n-11)}{n-3} + \frac{\epsilon}{2}\right]\left|\weyl_1\right|,\\
 \triangle e^{-b\rho} & \leq \left[\frac{1}{2} \frac{(n-1)(3n-11)}{n-3} + \frac{\epsilon}{2}\right] e^{-b\rho},
\end{aligned}
\right.
\]
and such that for any $W^{1,2}$-function $\phi$ supported in $K$, the following $L^2$-estimate holds (Lemma \ref{lmChengYau}):
\[\int_{X^n} |\nabla \phi|^2 dV_g \geq \left[\frac{(n-1)^2}{4} - \frac{\epsilon}{2}\right] \int \phi^2 dV_g.\]

We set $\psi = \left|\weyl_1\right| - C e^{-b\rho}$ where $C$ is chosen so large that $\psi < 0$ on $K$. Then $\psi$
satisfies
\begin{equation}\label{eqPsi}
 \triangle \psi \geq - \left[\frac{1}{2} \frac{(n-1)(3n-11)}{n-3} + \frac{\epsilon}{2}\right] \psi.
\end{equation}

We also define $\psi_+ = \max \{\psi, 0\}$ and note that $\psi_+ \in W^{1,2}$ and $\supp \psi_+ \subset X \setminus K$.
From Inequality \eqref{eqPsi}, we get
\begin{eqnarray*}
\left[\frac{(n-1)^2}{4} - \frac{\epsilon}{2}\right] \int_X \left|\psi_+\right|^2 dV_g
 & \leq & \int_{X^n} |\nabla \psi_+|^2 dV_g\\
 & \leq & - \int_{X^n} \psi_+ \triangle \psi dV_g\\
 & \leq & \left[\frac{1}{2} \frac{(n-1)(3n-11)}{n-3} + \frac{\epsilon}{2}\right] \int_{X^n} (\psi_+)^2 dV_g,\\
\left[\frac{1}{4} \frac{(n-1)(n-5)^2}{n-3}-\epsilon\right] \int_{X^n} (\psi_+)^2 dV_g
 & \leq & 0.
\end{eqnarray*}
From our assumption on $\epsilon$, this immediately implies that $\psi_+ \equiv 0$, that is to say
\[\left|\weyl_1\right| \leq C e^{-b\rho},\] or equivalently,
\[\left|\weyl\right| \leq C e^{-\left(\frac{1}{2} \frac{(n-1)^2}{n-3} + \delta\right)\rho}.\]

Since $n \geq 8$, $\frac{1}{2} \frac{(n-1)^2}{n-3} > 2$. In particular, from Proposition \ref{propHQS},
we conclude that the manifold $(X^n, g)$ is $C^{2, \mu}$-conformally compact for some $\mu \in (0; 1)$.
So it falls into the assumptions of Theorem \ref{curvaturebehavior}. This concludes the proof
of Theorem \ref{mainthm} for $n \geq 8$.

If $n = 7$, then $\frac{n-1}{n-3} = \frac{3}{2} > 1 = \frac{n-5}{2}$ so we can no longer apply Proposition \ref{lp}
for this value of $a$. Instead we choose $a = \frac{n-5}{2} - \frac{1}{4}$.
Inequality \eqref{eqNormWeyl3} becomes
\[
 \triangle |\weyl_1| \geq \left[8 + \frac{5}{8} + o(1)\right]\left|\weyl_1\right|.
\]
Setting $b \definedas 3 + \delta$, it can be checked that $e^{-b\rho}$ satisfies
\[
 \triangle e^{-b \rho} \leq \left[8 + \frac{5}{8} + o(1)\right] e^{-b\rho},
\]
outside some compact subset. We can then rephrase the previous proof, the only point to
note is that \[\frac{(n-1)^2}{4} = 9 > 8 + \frac{5}{8}.\] This is what allows the use of the asymptotic
$L^2$-estimate (Lemma \ref{lmChengYau}).
\end{proof}

%% file: applications.tex
Together with Theorem \ref{mainresult1}, the rigidity result \cite[Theorem 1.6]{HQS} implies

\begin{thm}\label{rigidity}
Suppose that $(X^n,g)$, $n \geq 7$ is a complete noncompact Einstein manifold with an essential set $\bD$ and that
$X^n$ is simply connected at infinity. If we further assume $\| \weyl\|_{L^p(X^n,g)} < \infty$ for some $p$ satisfying
$1< p <\frac{n-1}{2}$, then $(X^n,g)$ is isometric to $\mathbb{H}^n$.
\end{thm}

\begin{proof}
By Theorem \ref{mainthm}, we know that there exits a constant $C > 0$ such that
$$|\weyl| \leq Ce^{-(n+1)\rho}.$$

On the other hand, by a direct refinenemt of the proof of Lemma \ref{lmDecayWeyl}, we also have
$$|\nabla \weyl| \leq C' e^{-(n+1)\rho}$$ for some constant $C' > 0$. See also \cite[Theorem 4.3]{BG}.
The theorem then follows from \cite[Theorem 5.1]{HQS}.
\end{proof}

As another application, we consider a similar question for static vacuum spacetimes. We recall that an
$(n + 1)$-dimensional static spacetime $(N^{n+1}, g)$ is a solution of the vacuum Einstein equations
\[\ric - \frac{\scal}{2} g + \Lambda g = 0\]

of the form

\[
\begin{aligned}
N^{n+1} & = \bR \times M^n,\\
g & = -V ^2dt^2 + h
\end{aligned}
\]

where $(M^n, h)$ is a Riemannian manifold, $V$ is a positive function on $M^n$ and $\Lambda$ is the so-called
cosmological constant which we choose equal to $-\frac{n(n-1)}{2}$. The vacuum Einstein equations can be written
in terms of $h$ and $V$ as

\begin{equation}\label{ric}
\ric_h + n h = \frac{\hess(V)}{V},
\end{equation}
and
\begin{equation}\label{delta}
\triangle_h V = nV.
\end{equation}

Computing the trace of these two equations, we see that $h$ has constant scalar curvature $\scal = -n(n- 1)$.
We will often just call the triple $(M^n, h, V)$ a static vacuum. We set

\[
\begin{aligned}
 \mathrm{E} & \definedas \riem_h - \bfK,\\
 \mathrm{T} & \definedas \ric_h + (n-1) h = \frac{\hess(V)}{V} - h.
\end{aligned}
\]

As another application of Theorem \ref{mainresult1}, we state the following theorem:

\begin{thm}\label{static}
Suppose given $(M^n,h,V)$ a static vacuum with $n \geq 6$ such that $(M, h)$ has an essential set $\bD$ and
$\int_ {M^n}V |E|_h^p dV_h<\infty$, for some $p \in \left(1; \frac{n}{2}\right)$. If we further assume that
$\left\<dV, d\rho\right\> >0$ outside of $\bD$, then there exists a constant $C > 0$ such that
$$|E| \leq Ce^{-(n+2)\rho}$$
and
$$\left|\frac{\hess(V)}{V} - h\right| \leq Ce^{-(n+2)\rho},$$
where $\rho$ is the distance to the essential set $\bD$.
\end{thm}

\begin{rem}
The assumption $\left\<dV, d\rho\right\> > 0$ outside $\bD$ is natural and reasonable. Indeed, it is expected
that $V$ grows as $e^\rho$ at infinity. More precisely, $V$ is expected to have the following expansion at infinity:
$V = v e^\rho + O(1)$, where $v(x) = v(x^i)$ is a positive $C^2$ function depending only on the projection of a point
on the boundary of the essential set. The assumption $\left\<dV, d\rho\right\> > 0$ outside some compact set then can
be seen as a consequence of the fact that $V_0 = v e^{\rho}$ is a nice approximate solution of the equation
$\triangle V = n V$. We refer the reader to \cite{HS} and \cite{GicquaudCompactification} for more details.
\end{rem}

Notice that for static vacuum $(M^n,h,V)$, the Riemannian metric $g= V^2 d\theta^2 + h$ is an Einstein metric on
$\bS^{1}\times M$. Hence we consider the Einstein manifold $(\bS^1 \times M, V^2 d\theta^2 + h)$. For convenience,
in the following, the index $0$ refers to the direction $\partial_\theta$. Latin indices take values $1$ to $n$ and
refer to coordinates on $M$.

In order to prove Theorem \ref{static}, we need the following two lemmas:

\begin{lem}\label{WE}
Let $(M^n, h, V)$ be a static vacuum, if $\int_M V |E|_h^p dV_h < \infty$, then
$$\int_{\bS^1 \times M} |\weyl_g|^p_g dV_g < \infty.$$
\end{lem}

\begin{proof}
By a direct computation, using Equations \eqref{ric} and \eqref{delta}, we get

\[
\begin{aligned}
W_{ijkl}(g) & = E_{ijkl},\\
W_{0jkl}(g) & = 0,\\
W_{0j0l}(g) & = - V^2(V^{-1}\nabla_{h}^{2}V-h)=-V^2T,
\end{aligned}
\]
and
$$|\weyl(g)|^2_g= |E|^2_h + 4 |T|^2_h.$$

Note that $T_{ik}=h^{jl}E_{ijkl}$, thus there is a constant $C = C(n)$ such that
$$|E|_h \leq |\weyl_g|_g \leq C |E|_h.$$

Therefore the assumption \[\int_ {M^n}V |E|_h^{p} dV_h < \infty\] is equivalent to
\[\int_{{S}^{1}\times M^n} |\weyl_g|^{p}_g dV_g < \infty.\]
\end{proof}

\begin{lem}\label{essential}
Let $(M^n, h, V)$ be a static vacuum. If $(M^n, h)$ has an essential set $\bD$, $\rho$ is the distance to $\bD$ and $\left\<d\rho, dV\right\> >0$
outside $\bD$, then the manifold $(\bS^1 \times M^n, g)$ admits an essential set.
\end{lem}

\begin{proof}
From \cite[Lemma 2.5.11]{GicquaudThesis}, the existence of an essential set is a consequence of the following two facts:
\begin{enumerate}
 \item $\sec_g < 0$ outside a compact subset $K \subset \bS^1 \times M$,
 \item there exists a proper smooth function $f$ whose Hessian is positive definite outside $K$.
\end{enumerate}

We note that the assumption $\left\<d\rho, dV\right\> >0$ outside $\cD$ implies in particular that $\inf V = \inf_{\cD} V$ ($V$ grows along the
gradient lines of $\rho$). This implies that the metric $g$ has injectivity radius bounded from below. Then, mimicking the proof of Lemma \ref{lmDecayWeyl}
and using Lemma \ref{WE}, we get that $|\weyl_g| \to 0$ at infinity. This proves the first point.

Next we extend $\rho$ to $\bS^1 \times M$ by making it constant along the circles $\bS^1$. The Hessian of $\rho$ can be computed explicitely:
\[
\begin{aligned}
 \nabla_{ij}^{(g)} \rho & = \nabla_{ij}^{(h)} \rho,\\
 \nabla_{0i}^{(g)} \rho & = 0,\\
 \nabla_{00}^{(g)} \rho & = V \left\<d\rho, dV\right\>.
\end{aligned}
\]
It is then straightforward to see from the assumptions that $\hess^g(\rho)$ is positive definite outside $\cD$. This proves the lemma.
\end{proof}

Theorem \ref{static} is then a consequence of Theorem \ref{mainthm} applied to the metric $g = V^2 d\theta^2 + h$.

If we further assume that $M$ is spin, then we fall into the assumptions of \cite[Theorem 1]{W2} (See also Theorem 1.2 in \cite{Q}) so we get
the following theorem:

\begin{thm}\label{thmAdS}
Suppose that $(M^n,h,V)$ is a static vacuum with $n \geq 6$. Assume further that
\begin{enumerate}
 \item $M$ is spin,
 \item $(M, h)$ has an essential set $\bD$,
 \item $\left\<d\rho, dV\right\> >0$ outside $\bD$,
 \item and $\int_{M^n}V |E|_h^p dV_h < \infty$ for some $p \in \left(1; \frac{n}{2}\right)$
\end{enumerate}
then $(M^n,h)$ is the hyperbolic space and $V = \cosh(r)$, where $r$ is the distance function to a certain point $x_0 \in M$.
Equivalently, the spacetime $(\bR \times M, -V^2 dt^2 + h)$ is the anti-deSitter space.
\end{thm}